\documentclass[10pt]{amsart}
\usepackage{amssymb, amsthm, amsmath}
\usepackage[T1]{fontenc}
\usepackage[all]{xy}
\usepackage{graphicx}
\usepackage{psfrag}

\newtheorem{prop}{Proposition}[section]
\newtheorem{thm}[prop]{Theorem}
\newtheorem{cor}[prop]{Corollary}
\newtheorem{lemma}[prop]{Lemma}
\theoremstyle{definition}

\newtheorem{defn}[prop]{Definition}
\newtheorem{example}[prop]{Example}

\numberwithin{equation}{section}

\newcommand\A{{\mathcal A}}

\newcommand\C{{\mathbb C}}
\newcommand\Q{{\mathbb Q}}
\newcommand\N{{\mathbb N}}
\newcommand{\cN}{{\mathcal N}}
\newcommand\IH{{\mathbb H}}

\newcommand{\Ti}{\Theta}

\newcommand{\om}{{\varpi}}

\newcommand\cC{{\mathcal C}}

\newcommand\cH{{\mathcal H}}

\newcommand\X{{\mathfrak X}}
\newcommand\x{{\mathrm{x}}}
\newcommand\y{{\mathrm{y}}}

\newcommand\Z{{\mathbb Z}}

\newcommand\cW{{\mathcal W}}

\newcommand\AS{{\mathfrak S}}
\newcommand\BS{{\mathfrak B}}
\newcommand\CS{{\mathfrak C}}
\newcommand\DS{{\mathfrak D}}
\newcommand\XX{{\mathrm X}}
\newcommand\YY{{\mathrm Y}}
\newcommand\RR{R^{\infty}}
\newcommand\I{{\mathrm I}}
\newcommand\J{{\mathrm J}}

\newcommand\fraka{{\mathfrak a}}
\newcommand\frakb{{\mathfrak b}}
\newcommand\cc{{\mathfrak c}}

\newcommand\al{\alpha}
\newcommand{\be}{\beta}
\newcommand\la{\lambda}
\newcommand\s{{\sigma}}
\newcommand\Eta{H}
\newcommand\supp{{\mathrm{supp}}}

\newcommand\noin{\noindent}
\newcommand\lra{\longrightarrow}

\newcommand\bull{{\scriptscriptstyle \bullet}}
\newcommand\eqto{\stackrel{\lower1.5pt\hbox{$\scriptstyle\sim\,$}}\to}
\newcommand\ov{\overline}

\newcommand\wh{\widehat}
\newcommand\wt{\widetilde}
\newcommand\dis{\displaystyle}

\DeclareMathOperator{\Pf}{Pfaffian}

\DeclareMathOperator{\Sp}{Sp}

\DeclareMathOperator{\GL}{GL}
\DeclareMathOperator{\LG}{LG}

\DeclareMathOperator{\IF}{IF}

\DeclareMathOperator{\HH}{\mathrm{H}}

\DeclareMathOperator{\type}{\mathrm{type}}
\DeclareMathOperator{\rank}{\mathrm{rank}}

\newcommand{\ignore}[1]{}

\begin{document}

\title[Schubert polynomials and degeneracy locus formulas]
{Schubert polynomials and degeneracy locus formulas}

\date{August 31, 2017}

\author{Harry~Tamvakis} \address{University of Maryland, Department of
Mathematics, William E. Kirwan Hall, 4176 Campus Drive, 
College Park, MD 20742, USA}
\email{harryt@math.umd.edu}


\thanks{The author was supported in part by NSF Grant DMS-1303352.}

\begin{abstract}
In previous work \cite{T6}, we employed the approach to Schubert
polynomials by Fomin, Stanley, and Kirillov to obtain simple, uniform
proofs that the double Schubert polynomials of Lascoux and
Sch\"utzenberger and Ikeda, Mihalcea, and Naruse represent degeneracy
loci for the classical groups in the sense of Fulton. Using this as
our starting point, and purely combinatorial methods, we obtain a new
proof of the general formulas of \cite{T5}, which represent the
degeneracy loci coming from any isotropic partial flag variety. 
Along the way, we also find several new formulas and elucidate the
connections between some earlier ones.
\end{abstract}

\maketitle

\tableofcontents

\setcounter{section}{-1}

\section{Introduction}

In the 1990s, Fulton \cite{Fu1, Fu2} introduced a notion of degeneracy
loci determined by flags of vector bundles associated to any classical
Lie group $G$. For any two (isotropic) flags $E_\bull$ and $F_\bull$
of subbundles of a (symplectic or orthogonal) vector bundle $E$ over a
base variety $M$, and an element $w$ in the Weyl group of $G$, there
is a locus $\X_w\subset M$ defined by incidence relations between the
flags.  The {\em degeneracy locus problem} is to find a universal
polynomial $P_w$ in the Chern classes of the vector bundles involved
such that $[\X_w] = P_w\cap [M]$. We ask that $P_w$ should be
combinatorially explicit and manifestly respect the symmetries (that
is, the descent sets) of both $w$ and $w^{-1}$, whenever possible.

When $G$ is the general linear group, the degeneracy locus problem was
solved by Buch, Kresch, Yong, and the author \cite{BKTY}. This answer
was extended in a type uniform way to the symplectic and orthogonal
Lie groups in \cite{T5}. The formulas of \cite{BKTY, T5} rely in part
on a theory of {\em Schubert polynomials}, which express the classes
of the degeneracy loci in terms of the Chern {\em roots} of the vector
bundles $E_\bull$ and $F_\bull$. The desired combinatorial theory of
Schubert polynomials, together with its connection to geometry, was
established in the papers \cite{LS, L, Fu1} (for Lie type A) and
\cite{BH, T2, T3, IMN1, T5} (for Lie types B, C, and D).

In previous work \cite[\S 7.3]{T6}, we employed Fomin, Stanley, and
Kirillov's nilCoxeter algebra approach to Schubert polynomials
\cite{FS, FK} to give simple, uniform proofs that the double Schubert
polynomials of Lascoux and Sch\"utzenberger \cite{LS, L} and Ikeda,
Mihalcea, and Naruse \cite{IMN1} represent degeneracy loci of vector
bundles, in the above sense. Our main goal in this paper is to begin
with the same definition of Schubert polynomials from \cite{T5, T6} and, by
purely combinatorial methods, derive the splitting formulas for these
polynomials found in \cite[\S 3 and \S 6]{T5}. The latter results then 
imply the general degeneracy locus formulas of \cite{T5}. The proof of the 
corresponding type A splitting formula from \cite{BKTY} is essentially
combinatorial; this is clarified in \cite[\S 1.4]{T5} and \S \ref{tAt}
of the present paper.

As in \cite{T5}, our arguments depend on two key results from
\cite{BKT2, BKT3}, which state that the single Schubert polynomials
indexed by Grassmannian elements of the Weyl groups are represented by
(single) {\em theta} and {\em eta polynomials}. The original proofs of
these theorems used the classical Pieri rules from \cite{BKT1}, which
were derived geometrically in op.\ cit.\ by intersecting Schubert
cells. More recent proofs by Ikeda and Matsumura \cite{IM}, the
author and Wilson \cite{TW, T7}, and Anderson and Fulton \cite{AF2}
use localization in equivariant cohomology (following \cite{Ar, KK})
or employ other geometric arguments stemming from Kazarian's work
\cite{Ka}. However, the statements of the aforementioned theorems from
\cite{BKT2, BKT3} are entirely combinatorial, and it is natural to
seek proofs of them within the same framework. The corresponding
result in type A is the elementary fact that the symmetric Schubert
polynomials are equal to Schur polynomials.

The approach we take here begins by extending Billey and Haiman's
formula \cite[Prop.\ 4.15]{BH} for the single Schubert polynomials
indexed by the longest element in the Weyl group of $G$ to the double
Schubert polynomials. From this, we deduce corresponding
Pfaffian formulas equivalent to those in \cite[Thm.\ 1.2]{IMN1}
and \cite{AF1, AF2} for the (equivariant) Schubert class of a point on
the complete flag variety $G/B$. Following the method of \cite[\S
  8]{IM} in the symplectic case, by employing the left divided
difference operators, we then derive analogous formulas for the class
of a point on symplectic and even orthogonal Grassmannians, which
first appeared in Kazarian's paper \cite{Ka}. The proof continues by
using the arguments found in \cite{TW, T7} to arrive at {\em double
  theta} and {\em double eta polynomials}, and then specializing to
obtain their single versions in \cite{BKT2, BKT3}. Finally, we
establish more general versions of the Schubert splitting
formulas of \cite{T5} with the help of the double mixed Stanley
functions and $k$-transition trees introduced in op.\ cit.

The result of our efforts is a straightforward and type uniform
combinatorial proof of the equivariant Giambelli and degeneracy locus
formulas of \cite{Ka, I, IN, IMN1, IM, T5, T7, TW}. Along the way, we
obtain several new formulas, and illuminate the connections between
some earlier ones. In particular, we define in \S \ref{AStil} the {\em
  reverse double Schubert polynomials} $\wt{\AS}_\om(Y,Z)$, which
provide a bridge between the type A and type C theories. Furthermore,
in type D, inspired by \cite{I, IN, IMN1}, we use Ivanov's double (or
factorial) Schur $P$-functions \cite{Iv1, Iv2, Iv3}, and obtain in \S
\ref{spfan} new results about them, which are suitable for our
purposes.

Our earlier paper \cite{T6} contained a detailed exposition of the
various ingredients that went into \cite{T5}, revealing the author's
perspective on the subject of degeneracy loci, as it stood in 2009. We
revisit some of that material here, for completeness, but focus in \S
\ref{tAt} -- \S \ref{tDt} on what is required for our new
combinatorial proofs of the main results. In \S \ref{geometriz}, we
discuss how to use {\em geometrization} to translate the theorems of
the previous sections into Chern class formulas, and provide some
detailed remarks on the history of this problem, which supplement the
ones contained in \cite{T6}.

In a recent preprint, Anderson and Fulton \cite{AF2} use Young's
raising operators and algebro-geometric arguments to define {\em
  multi-theta} and {\em multi-eta polynomials}, which extend the
double theta and eta polynomials of \cite{TW, T7, W} even further. The
resulting degeneracy locus formulas and their proofs are an important
contribution to the theory of theta and eta polynomials, but hold only
for certain special elements $w$ of the Weyl group of $G$. We leave
the task of including them within the present algebraic and
combinatorial framework to future research.

As we mentioned above, besides new and uniform proofs of earlier
theorems, the present paper also contains many original results, which
appear in the text without attribution. For the reader's convenience,
we provide a list of the main ones here. To the best of our knowledge,
Definition \ref{rdsp}, Propositions \ref{SchubOm}, \ref{Ctopfirst},
\ref{wknprop2}, \ref{split2prop}, \ref{Pxtprop}, \ref{Pprop},
\ref{Dtopfirst}, and \ref{split2Dprop}, and Corollaries
\ref{Schubom0}, \ref{Pcor}, and \ref{Dtopcor} are new.

This article is organized as follows. In \S \ref{tAt} we review the
type A theory, culminating in the relevant splitting formulas for type
A double Schubert polynomials from \cite{BKTY}. In \S \ref{fspdual},
we discuss flagged Schur polynomials, type A duality, and introduce
the reverse double Schubert polynomials $\wt{\AS}_\om(Y,Z)$. These
objects are used in \S \ref{tCt} and \S \ref{tDt}, which provide the
corresponding theory in Lie types C and D.  Finally, in \S
\ref{geometriz} we discuss the history of the geometrization of the
single and double Schubert polynomials, focusing on the symplectic
case.

My work on this paper was inspired by my participation in the
conference in Algebraic Geometry `IMPANGA 15' on 12-18 April 2015 in
B\k{e}dlewo, Poland. I am grateful to the organizing committee for the
invitation and for bringing together many researchers in the area in
such a pleasant and stimulating atmosphere. I also thank Andrew Kresch
for his helpful comments on the manuscript.

\section{The type A theory}
\label{tAt}

\subsection{Schubert polynomials and divided differences}
\label{spsdd}

The Weyl group $S_n$ of permutations of $\{1,\ldots, n\}$ is generated
by the $n-1$ simple transpositions $s_i=(i,i+1)$. 
A {\em reduced word} of a permutation $\om$
in $S_\infty$ is a sequence $a_1\cdots a_\ell$ of positive integers
such that $\om=s_{a_1}\cdots s_{a_\ell}$ and $\ell$ is minimal, so (by
definition) equal to the length $\ell(\om)$ of $\om$. We say that
$\om$ has a {\em descent} at position $r$ if $\ell(\om s_r)<\ell(\om)$,
where $s_r$ is the simple reflection indexed by $r$.

Our main references for type A Schubert polynomials are \cite{LS, M2,
  FS}.  We recall from \cite{FS} their construction using the {\em
  nilCoxeter algebra} $\cN_n$ of the symmetric group $S_n$.  By
definition, $\cN_n$ is the free associative algebra with unit
generated by the elements $u_1,\ldots,u_{n-1}$ modulo the relations
\[
\begin{array}{rclr}
u_i^2 & = & 0 & i\geq 1\ ; \\
u_iu_j & = & u_ju_i & |i-j|\geq 2\ ; \\
u_iu_{i+1}u_i & = & u_{i+1}u_iu_{i+1} & i\geq 1.
\end{array}
\]
For any $\om\in S_n$, choose a reduced word $a_1\cdots a_\ell$ for
$\om$ and define $u_\om := u_{a_1}\ldots u_{a_\ell}$. Then the $u_\om$
for $\om\in S_n$ are well defined and form a free $\Z$-basis of
$\cN_n$. We denote the coefficient of $u_\om\in \cN_n$ in the
expansion of the element $\xi\in \cN_n$ by $\langle \xi,\om\rangle$;
in other words, $\xi = \sum_{\om\in S_n}\langle \xi,\om\rangle\,u_\om$,
for all $\xi\in \cN_n$.

Let $t$ be an indeterminate and define
\begin{gather*}
A_i(t) := (1+t u_{n-1})(1+t u_{n-2})\cdots 
(1+t u_i) \ ; \\
\tilde{A}_i(t) := (1-t u_i)(1-t u_{i+1})\cdots (1-t u_{n-1}).
\end{gather*}
Suppose that $Y=(y_1,y_2,\ldots)$ and $Z=(z_1,z_2,\ldots)$ are two
infinite sequences of commuting independent variables.  For any $\om
\in S_n$, the {\em double Schubert polynomial} $\AS_\om(Y,Z)$ of
Lascoux and Sch\"utzenberger \cite{LS, L} is given by the prescription
\begin{equation}
\label{dbleA}
\AS_\om(Y,Z) :=\left\langle 
\tilde{A}_{n-1}(z_{n-1})\cdots \tilde{A}_1(z_1)
A_1(y_1)\cdots A_{n-1}(y_{n-1}), \om\right\rangle.
\end{equation}
The polynomial $\AS_\om(Y):=\AS_\om(Y,0)$ is the single Schubert
polynomial. The definition (\ref{dbleA}) implies that $\AS_\om(Y)$
has nonnegative integer coefficients, which admit a combinatorial 
interpretation (compare with \cite[Thm.\ 1.1]{BJS}).

For each $n\geq 1$, there is an injective group homomorphism $i_n: S_n
\hookrightarrow S_{n+1}$, defined by adjoining the fixed point $n+1$,
and we let $S_\infty :=\cup_n S_n$.  The polynomials $\AS_\om(Y,Z)$
have an important {\em stability property} under the inclusions $i_n$
of the Weyl groups, namely, if $\om\in S_n$, then we
have $$\AS_{i_n(\om)}(Y,Z) = \AS_\om(Y,Z).$$ The stability property
implies that $\AS_\om(Y,Z)$ is well defined for all $\om \in S_\infty$.

For any $i\geq 1$, there are {\em right} and {\em left} divided
difference operators $\partial_i^y$ and $\partial_i^z$ which act on
the polynomial ring $\Z[Y,Z]$. We define an action of $S_\infty$ on
$\Z[Y,Z]$ by ring automorphisms by letting the simple transpositions
$s_i$ act by interchanging $y_i$ and $y_{i+1}$ and leaving all the
remaining variables fixed. Define $\partial_i^y$ on $\Z[Y,Z]$ by
\[
\partial_i^yf := \frac{f-s_if}{y_i-y_{i+1}}.
\]
Consider the ring involution $\omega:\Z[Y,Z]\to\Z[Y,Z]$ determined by
$\omega(y_j) = -z_j$ and $\omega(z_j) = -y_j$ for each $j$, and set
$\partial_i^z:=\omega\partial_i^y\omega$ for each $i\geq 1$.

Both the single and double Schubert polynomials may be characterized
by their compatibility with the divided difference operators. In fact,
the polynomials $\AS_\om$ for $\om\in S_{\infty}$ are the unique
family of elements of $\Z[Y,Z]$ satisfying the equations
\begin{equation}
\label{peqA}
\partial_i^y\AS_\om = \begin{cases}
\AS_{\om s_i} & \text{if $\ell(\om s_i)<\ell(\om)$}, \\ 
0 & \text{otherwise},
\end{cases}
\quad
\partial_i^z\AS_\om = \begin{cases}
\AS_{s_i\om} & \text{if $\ell(s_i\om)<\ell(\om)$}, \\ 
0 & \text{otherwise},
\end{cases}
\end{equation}
for all $i\geq 1$, together with the condition that the constant term
of $\AS_\om$ is $1$ if $\om=1$, and $0$ otherwise. This
characterization theorem is straightforward to prove directly from the
definition (\ref{dbleA}), following \cite[Thm.\ 2.2]{FS}.

The above result has two important consequences. The first is that 
$\AS_\om(Y,Z)$ is symmetric in $y_i$, $y_{i+1}$ (respectively in
$z_j$, $z_{j+1}$) if and only if $\om_i<\om_{i+1}$ (respectively
$\om^{-1}_j<\om^{-1}_{j+1}$). In other words, the descents of $\om$
and $\om^{-1}$ determine the symmetries of the polynomial
$\AS_\om(Y,Z)$. A second consequence is that the
double Schubert polynomials $\AS_\om(Y,Z)$ represent the universal
Schubert classes in type A flag bundles, and therefore degeneracy loci
of vector bundles, in the sense of \cite{Fu1}. 

Let $\om_0=(n,n-1,\ldots,1)$ be the longest element of $S_n$. 
According to \cite[Lemma 2.1]{FS}, for all commuting variables $s$, $t$ 
and indices $i$, we have $A_i(s)A_i(t) = A_i(t)A_i(s)$. Since 
$\tilde{A}_i(t)=A_i(t)^{-1}$, we also have $A_i(s)\tilde{A}_i(t)=
\tilde{A}_i(t)A_i(s)$. Fomin
and Stanley \cite[Cor.\ 4.4]{FS} use this fact and the definition 
(\ref{dbleA}) to show that
\begin{equation}
\label{prodeq}
\AS_{\om_0}(Y,Z) = \prod_{i+j \leq n} (y_i-z_j).
\end{equation}

\subsection{Schur polynomials}
\label{examps}

Following \cite{LS, M2}, the product in (\ref{prodeq}) may be written
in the form of a multi-Schur determinant. Furthermore, by applying
divided differences to $\AS_{\om_0}$ and using the equations
(\ref{peqA}), one can express more general Schubert polynomials
$\AS_\om$ as Schur type determinants. We will not reprove these
formulas here, but we do need some more notation to recall the ones
that we will require.

For any integer $j\geq 0$, define the
elementary and complete symmetric functions 
$e_j(Y)$ and $h_j(Y)$ by the generating series
\[
\prod_{i=1}^{\infty}(1+y_it) = \sum_{j=0}^{\infty}
e_j(Y)t^j \ \ \ \text{and} \ \ \ 
\prod_{i=1}^{\infty}(1-y_it)^{-1} = \sum_{j=0}^{\infty}
h_j(Y)t^j,
\]
respectively. We define the supersymmetric functions $h_p(Y/Z)$ for
$p\in \Z$ by the generating function equation
\[
\sum_{p=0}^\infty h_p(Y/Z)t^p = \left(\sum_{j=0}^{\infty}h_j(Y)t^j\right)
\left(\sum_{j=0}^{\infty}e_j(Z)(-t)^j\right).
\]
If $r\geq 1$ then we let $e^r_j(Y):=e_j(y_1,\ldots,y_r)$ 
and $h^r_j(Y):=h_j(y_1,\ldots,y_r)$ denote the polynomials
obtained from $e_j(Y)$ and $h_j(Y)$ by setting $y_j=0$ for all $j>r$. Let 
$e^0_j(Y)=h^0_j(Y)=\delta_{0j}$, where $\delta_{0j}$ denotes the
Kronecker delta, and for $r<0$, define $h^r_j(Y):=e^{-r}_j(Y)$
and $e^r_j(Y):=h^{-r}_j(Y)$.

We will work with integer sequences $\al = (\al_1,\al_2,\ldots)$ which
are assumed to have finite support when they appear as subscripts.
The sequence $\al$ is a {\em composition} if $\al_i\geq 0$ for all
$i$, and a {\em partition} if $\al_i\geq \al_{i+1}\geq 0$ for all
$i\geq 1$. We set $|\al|:=\sum_i\al_i$. Partitions are traditionally
identified with their Young diagram of boxes, and this is used to 
define the inclusion relation $\mu\subset\la$ between two partitions
$\mu$ and $\la$.

Given an integer sequence $\al$, we define 
$s_\al(Y/Z)$ by the determinantal equation
\begin{equation}
\label{deteq}
s_\al(Y/Z):= \det(h_{\al_i+j-i}(Y/Z))_{i,j}.
\end{equation}
Notice that the matrix $\{h_{\al_i+j-i}(Y/Z)\}_{i,j}$ is upper
unitriangular for $i$ and $j$ sufficiently large, so the determinant
in (\ref{deteq}) is well defined. When $\al=\la$ is a partition, then
$s_\la(Y/Z)$ is called a {\em supersymmetric Schur function}.  The
usual Schur $S$-function $s_\la(Y)$ satisfies
$s_\la(Y):=s_{\la}(Y/Z)\vert_{Z=0}$. Moreover, we have
$s_\la(0/Z)=s_\la(Y/Z)\vert_{Y=0}=(-1)^{|\la|}s_{\la'}(Z)$, where
$\la'$ denotes the conjugate (or transpose) partition of $\la$.
Observe that $s_\la(y_1,\ldots,y_m)=0$ if the {\em length} $\ell(\la)$,
that is, the number of non-zero parts $\la_i$, is greater than $m$.

For any positive integers $i<j$ and integer sequence $\alpha$, define
the Young raising operator $R_{ij}$ by $R_{ij}(\alpha) :=
(\alpha_1,\ldots,\alpha_i+1,\ldots,\alpha_j-1, \ldots)$. Using these
operators, the equation (\ref{deteq}) may be rewritten as
\[
s_\al(Y/Z) = \prod_{i<j}(1-R_{ij}) \, h_\al(Y/Z),
\]
where $h_\al:= \prod_i h_{\al_i}$ and each operator $R_{ij}$ acts on
the expression $h_\al$ (regarded as a noncommutative monomial) by the
prescription $R_{ij} h_\al:= h_{R_{ij}\al}$. See \cite{T4} for more
information on raising operators.

A permutation
$\om\in S_\infty$ is {\em Grassmannian} if there exists an $m\geq 1$
such that $\om_i<\om_{i+1}$ for all $i\neq m$.  The {\em shape} of
such a Grassmannian permutation $\om$ is the partition
$\la=(\la_1,\ldots, \la_m)$ with $\la_{m+1-j}=\om_j-j$ for $1\leq j
\leq m$. If $\om\in S_\infty$ is a Grassmannian permutation with a unique
descent at $m$ and shape $\la$, then we have
\begin{equation}
\label{StoS}
\AS_\om(Y)=s_\la(y_1,\ldots,y_m). 
\end{equation}
A short proof of (\ref{StoS}) starting from the formula (\ref{prodeq}) for
$\AS_{\om_0}(Y)$ is in \cite[(4.8)]{M2}.

\subsection{Stanley symmetric functions and splitting formulas}
\label{ssfsfsA}

Given any permutations $u_1,\ldots,u_p,\om$, we will write
$u_1\cdots u_p=\om$ if $\ell(u_1)+\cdots + \ell(u_p)=\ell(\om)$ and the
product of $u_1,\ldots,u_p$ is equal to $\om$. In this case we say that
$u_1\cdots u_p$ is a {\em reduced factorization} of $\om$. Equation
(\ref{dbleA}) impies the relation
\begin{equation}
\label{dbleA2}
\AS_\om(Y,Z) = \sum_{uv=\om}\AS_{u^{-1}}(-Z)\AS_v(Y)
\end{equation}
summed over all reduced factorizations $uv=\om$ in $S_\infty$.

If $A(Y):=A_1(y_1)A_1(y_2)\cdots$, then the function $G_\om(Y)$
defined for $\om\in S_n$ by
\[
G_\om(Y) := \langle A(Y), \om\rangle 
\]
is symmetric in $Y$. $G_\om$ is the type A {\em Stanley symmetric
  function}, which was introduced in \cite{S}.\footnote{In Stanley's
  paper, the function $G_{\om^{-1}}$ is assigned to $\om$.}  If
$\tilde{A}(Z):=\tilde{A}_1(z_1)\tilde{A}_1(z_2)\cdots$, then we define
the {\em double Stanley symmetric function} $G_\om(Y/Z)$ by
\[
G_\om(Y/Z) := \langle \tilde{A}(Z)A(Y), \om\rangle =
\sum_{uv=\om}G_{u^{-1}}(-Z)G_v(Y)
\]
with the sum over all reduced factorizations $uv=\om$ in $S_\infty$.

Given $m\geq 1$ and any $\om \in S_n$, the permutation $1_m\times
\om\in S_{m+n}$ is defined by $(1_m\times \om)(j)=j$ for $1\leq j \leq
m$ and $(1_m\times \om)(j)= m+\om(j-m)$ for $j>m$.  We say that a
permutation $\om$ is {\em increasing up to $m$} if $\om(1)<\om(2) <
\cdots < \om(m)$. If $\om$ is increasing up to $m$, then we have the
following key identity:
\begin{equation}
\label{keyid}
\AS_\om(Y)= \sum_{v(1_m\times u) = \om}G_v(y_1,\ldots,y_m)\AS_u(y_{m+1},y_{m+2},\ldots)
\end{equation}
where the sum is over all reduced factorizations $v(1_m\times u) = \om$
in $S_\infty$.  Equation (\ref{keyid}) admits a 
double version: let $G_v(Y_{(m)}/Z_{(\ell)})$ denote the polynomial
obtained from $G_v(Y/Z)$ by setting $y_i=z_j=0$ for 
all $i>m$ and $j>\ell$. Then if $\om$ is increasing up to $m$ and 
$\om^{-1}$ is increasing up to $\ell$, we have
\begin{equation}
\label{keyid2}
\AS_\om(Y,Z)= 
\sum \AS_{u^{-1}}(-Z_{>\ell}) G_v(Y_{(m)}/Z_{(\ell)})\AS_{u'}(Y_{>m}),
\end{equation}
where $Y_{>m}:=(y_{m+1},y_{m+2},\ldots)$,
$-Z_{>\ell}:=(-z_{\ell+1},-z_{\ell+2},\ldots)$, and the sum is over
all reduced factorizations $(1_{\ell}\times u)v(1_m\times u') = \om$ in
$S_\infty$. Equations (\ref{keyid}) and (\ref{keyid2}) are easy to show 
directly from the definitions of $\AS_\om(Y,Z)$ and $G_v(Y/Z)$ (see 
\cite[\S 1.4]{T5} for a detailed proof of (\ref{keyid}); the proof of 
(\ref{keyid2}) is similar).

We say that a permutation $\om\in S_\infty$ is {\em compatible} with
the sequence $\fraka \, :\, a_1 < \cdots < a_p$ of positive integers
if all descent positions of $\om$ are contained in $\fraka$.  Let
$\frakb \, :\, b_1 < \cdots <b_q$ be a second sequence of positive
integers and assume that $\om$ is compatible with $\fraka$ and
$\om^{-1}$ is compatible with $\frakb$. We say that a reduced
factorization $u_1\cdots u_{p+q-1} = \om$ is {\em compatible} with
$\fraka$, $\frakb$ if $u_j(i)=i$ whenever $j<q$ and $i\leq b_{q-j}$ or
whenever $j>q$ and $i \leq a_{j-q}$ (and where we set $u_j(0)=0$).
Set $Y_i := \{y_{a_{i-1}+1},\ldots,y_{a_i}\}$ for each $i\geq 1$ and
$Z_j := \{z_{b_{j-1}+1},\ldots,z_{b_j}\}$ for each $j\geq 1$.

\begin{prop}[\cite{BKTY}]
\label{split1A}
Suppose that $\om$ and $\om^{-1}$ are compatible with $\fraka$ and 
$\frakb$, respectively. Then the
Schubert polynomial $\AS_\om(Y,Z)$ satisfies
\[
\AS_\om = \sum
G_{u_1}(0/Z_q)\cdots G_{u_{q-1}}(0/Z_2)
G_{u_q}(Y_1/Z_1)G_{u_{q+1}}(Y_2) \cdots G_{u_{p+q-1}}(Y_p)
\]
summed over all reduced factorizations $u_1\cdots u_{p+q-1} = \om$
compatible with $\fraka$, $\frakb$. 
\end{prop}
\begin{proof}
The result follows easily by using (\ref{keyid2}) and iterating the 
identity (\ref{keyid}).
\end{proof}

When the Stanley symmetric function $G_\om$ is expanded in the basis
of Schur functions, one obtains a formula
\begin{equation}
\label{Geq}
G_\om(Y) = \sum_{\la\, :\, |\la| = \ell(\om)}a^\om_\la s_\la(Y)
\end{equation}
for some nonnegative integers $a^\om_\la$. There exist several different
combinatorial interpretations of these coefficients, for instance
using the {\em transition trees} of Lascoux and Sch\"utzenberger (see 
for example \cite[(4.37)]{M2}). One also knows that 
$a_\la^\om = a^{\om^{-1}}_{\la'}$.

\begin{thm}[\cite{BKTY}, Thm.\ 4]
\label{Asplitthm}
Suppose that $\om$ is compatible with $\fraka$ and $\om^{-1}$ is compatible
with $\frakb$. Then we have
\begin{equation}
\label{asplit}
\AS_\om = \sum_{\underline{\la}} a^\om_{\underline{\la}}\,
s_{\la^1}(0/Z_q)\cdots s_{\la^{q-1}}(0/Z_2)
s_{\la^q}(Y_1/Z_1)s_{\la^{q+1}}(Y_2) \cdots s_{\la^{p+q-1}}(Y_p)
\end{equation}
summed over all sequences of partitions
$\underline{\la}=(\la^1,\ldots,\la^{p+q-1})$, where 
\[
a^\om_{\underline{\la}} := \sum_{u_1\cdots u_{p+q-1} = \om}
a_{\la^1}^{u_1}\cdots a_{\la^{p+q-1}}^{u_{p+q-1}},
\]
summed over all reduced factorizations $u_1\cdots u_{p+q-1} = \om$
compatible with $\fraka$, $\frakb$.
\end{thm}
\begin{proof}
The result follows from Proposition \ref{split1A} by using 
the equation (\ref{Geq}).
\end{proof}

Equation (\ref{asplit}) generalizes the monomial positivity of the
Schubert polynomial $\AS_\om(Y,Z)$ from \cite{BJS, FS} to a
combinatorial formula which manifestly respects the descent sets of
$\om$ and $\om^{-1}$, and therefore exhibits the symmetries of
$\AS_\om$.  Moreover, the splitting formula is uniquely determined
once $\om$ and the compatible sequences $\fraka$ and $\frakb$ are
specified. The main geometric application of equation (\ref{asplit})
is that it directly implies a corresponding Chern class formula for
the type A degeneracy locus indexed by $\om$, with the symmetries
native to the partial flag variety associated to $\fraka$. For more
details on this, as well as examples of explicit computations of the
{\em splitting coefficients} $a^\om_{\underline{\la}}$, see \cite[\S
  4]{BKTY} and \cite[\S 4 and \S 6]{T6}.

\section{Flagged Schur polynomials and duality}
\label{fspdual}

\subsection{Flagged Schur polynomials}
\label{fsp}

Let $\al=\{\al_j\}_{1\leq j\leq \ell}$, $\be=\{\be_j\}_{1\leq j\leq
  \ell}$, and $\rho=\{\rho_j\}_{1\leq j\leq \ell}$ be three integer
sequences, and let $t=(t_1,t_2,\ldots)$ be a sequence of 
independent variables. The Schur type determinant
\[
S_{\al/\be}^\rho(h(t)):= 
\det \left(h^{\rho_i}_{\al_i-\be_j+j-i}(t)\right)_{1\leq i,j \leq \ell}
\]
is called a {\em flagged Schur polynomial}.
We define the polynomial $S_{\al/\be}^\rho(e(t))$ in a similar way.
The method of Gessel and Viennot \cite{GV} or Wachs \cite{Wa} shows that
for any partition $\la$ and increasing composition $\rho$, we have 
\begin{equation}
\label{monomialsum}
S^\rho_\la(h(t)) = \sum t^U
\end{equation}
where the sum is over all column strict Young tableaux $U$ of shape
$\la$ whose entries in the $i$-th row are $\leq \rho_i$ for all $i\geq 1$.

The following well known result will be used in \S \ref{1stsec}.

\begin{lemma}
\label{goodlem}
Let $\la$ and $\mu$ be two partitions of length at most $\ell$ with 
$\max(\la_1,\mu_1)\leq k$. Then we have
\begin{equation}
\label{goodeq}
\det\left(h_{\la_i-\mu_j+j-i}^{k+i-\la_i}(t)\right)_{1\leq i,j\leq \ell} = 
\det\left(e_{\la'_i-\mu'_j+j-i}^{k+\la'_i-i}(t)\right)_{1\leq i,j\leq k} .
\end{equation}
\end{lemma}
\begin{proof}
The argument follows the one in \cite[I.2, eq.\ (2.9)]{M1}. Let
$N:=k+\ell$ and define the matrices 
\[
A:=\left(h_{i-j}^{-i}(-t)\right)_{0\leq i,j \leq N-1} \ \ \text{and} \ \ 
B:=\left(h_{i-j}^{j+1}(t)\right)_{0\leq i,j \leq N-1}\,.
\]
It is well known that $A$ and $B$ are inverse to each other (see for
example \cite[I.3, Ex.\ 21]{M1}). Therefore each minor of $A$ is equal
to the complementary cofactor of $B^t$, the transpose of $B$. For the
minor of $A=(e_{i-j}^i(-t))$ with row indices $\la_i'+k-i$ ($1\leq
i\leq k$) and column indices $\mu_j'+k-j$ ($1\leq j \leq k$), the
complementary cofactor of $B^t=(h_{j-i}^{i+1}(t))$ has row indices
$k-1+i-\la_i$ ($1\leq i \leq \ell$) and column indices $k-1+j-\mu_j$
($1\leq j \leq \ell$). The equality (\ref{goodeq}) follows by taking
determinants.
\end{proof}

\subsection{The duality involution}

For any $r\geq 1$, let $\delta_r$ denote the partition $(r,r-1,\ldots, 1)$. 
Let $Y_{(n)}:=(y_1,\ldots,y_n)$ and $I_n\subset \Z[Y_{(n)}]$ be the ideal
generated by the elementary symmetric polynomials $e_i(Y_{(n)})$ for 
$1\leq i\leq n$. Set $H_n:=\Z[Y_{(n)}]/I_n$ and let $\cH_n$ be the 
$\Z$-linear subspace of $\Z[Y_{(n)}]$ spanned by the monomials 
$y^\al:=y_1^{\al_1}\cdots y_{n-1}^{\al_{n-1}}$ for all compositions
$\al\leq \delta_{n-1}$. 

Let ${\mathfrak Y}=\GL_n/B$ be the variety which parametrizes
complete flags $$E_\bull\ :\ 0\subsetneq E_1\subsetneq \cdots
\subsetneq E_n=\C^n$$ of subspaces of $\C^n$. It is well known that the
inclusion $\cH_n\subset \Z[Y_{(n)}]$ induces an isomorphism of abelian
groups $\cH_n\stackrel{\sim}\to H_n$, and that $H_n$ is
isomorphic to the cohomology ring of the flag variety ${\mathfrak
  Y}$. Moreover, the Schubert polynomials $\AS_\om(Y)$ for $\om \in
S_n$ form a $\Z$-basis of $\cH_n$, and represent the Schubert classes
in ${\mathfrak Y}$ under the above isomorphism (see for example \cite[\S
  10]{Fu3}).

Define an involution $\om\mapsto \om^*$ of $S_n$ by setting
$\om^*:=\om_0 \om \om_0$, and define the {\em duality involution}
$*:\cH_n\to\cH_n$ to be the $\Z$-linear map determined by
$*\,\AS_\om(Y) := \AS_{\om^*}(Y)$, for each $\om\in S_n$. We let
$\Z[Y_{(n)}]\to \Z[Y_{(n)}]$ be the ring involution given by
$(y_1,\ldots,y_n)\mapsto (-y_n,\ldots,-y_1)$, which descends to a ring
involution $D:H_n \to H_n$. The following result states that $*$ and
$D$ are mutually compatible.

\begin{lemma}
  \label{Dcomp}
For any permutation $\om\in S_n$, we have
\begin{equation}
\label{DAeq}
D(\AS_\om(Y)+I_n) = \AS_{\om^*}(Y)+I_n 
\end{equation}
in the quotient ring $H_n$.
\end{lemma}
\begin{proof}
Modulo the defining 
relations for the ideal $I_n$ of $\Z[Y_{(n)}]$, we have
\begin{equation}
\label{Aeq}
A_1(y_1)A_1(y_2)\cdots A_1(y_n) \equiv 1.
\end{equation}
Set 
\[
B_i(t):=(1+tu_{n-i})(1+tu_{n-i-1}t)\cdots (1+tu_1)
\]
so that 
\[
B_i(t)^{-1}=(1-tu_1)(1-tu_2)\cdots (1-tu_{n-i}).
\]
Observe that (\ref{Aeq}) may be written as
\[
A_1(y_1)A_2(y_2)\cdots A_{n-1}(y_{n-1})\cdot
B_{n-1}(y_2)B_{n-2}(y_3)\cdots B_1(y_n) \equiv 1
\]
or equivalently
\[
A_1(y_1)A_2(y_2)\cdots A_{n-1}(y_{n-1}) \equiv
B_1(y_n)^{-1} B_2(y_{n-1})^{-1}\cdots B_{n-1}(y_2)^{-1}.
\]
Notice that $(s_i)^*= s_{n-i}$ for each simple transposition $s_i$ in $S_n$.
It follows that for any $\om\in S_n$, we have
\begin{align*}
\AS_\om(Y) &= \langle A_1(y_1)A_2(y_2)\cdots A_{n-1}(y_{n-1}),\om  \rangle \\
&\equiv \langle B_1(y_n)^{-1} B_2(y_{n-1})^{-1}\cdots B_{n-1}(y_2)^{-1},\om\rangle
= D\, \AS_{\om^*}(Y),
\end{align*}
as required.
\end{proof}

Geometrically, the above involutions correspond to the duality
isomorphism ${\mathfrak Y}\to {\mathfrak Y}$ which sends each complete
flag $E_\bull$ in $\C^n$ to the dual flag $E'_\bull$ in
$(\C^n)^*$. Here the subspace $E'_i$ is defined as the kernel of the
canonical linear map $(\C^n)^*\to E^*_{n-i}$, for $1\leq i \leq
n$. For more details on this, see \cite[\S 10, Exercise 13]{Fu3}.

\begin{example}
One can prove the equality (\ref{DAeq}) directly
in the case when $\AS_\om(Y)$ is an elementary symmetric polynomial
$e_i(y_1,\ldots,y_r)$.  We have $\prod_{i=1}^n(1+y_it)=1$ in $H_n[t]$,
and hence
\[
\prod_{i=1}^r(1+y_it)\equiv\prod_{j=r+1}^n\frac{1}{1+y_jt}
\]
which implies that $e_i(y_1,\ldots,y_r) =
(-1)^ih_i(y_{r+1},\ldots,y_n)$ in $H_n$, for any $i,r\in [1,n]$. 
Applying the automorphism $D$ gives 
\begin{equation}
\label{Deq}
D(e_i(y_1,\ldots,y_r)+I_n) = h_i(y_1,\ldots,y_{n-r}) +I_n.
\end{equation}
\end{example}

\subsection{Reverse double Schubert polynomials}
\label{AStil}

\begin{defn}
\label{rdsp}
For any $\om\in S_\infty$, define the {\em reverse double 
Schubert polynomial}
\begin{equation}
\label{rdspeq}
\wt{\AS}_\om(Y,Z):= \sum_{uv=\om}\AS_u(Y)\AS_{v^{-1}}(-Z),
\end{equation}
where the sum is over all reduced factorizations $uv=\om$.
\end{defn}
The adjective `reverse' in Definition \ref{rdsp} is justified by
comparing (\ref{rdspeq}) with formula (\ref{dbleA2}). Observe also
that if $\om\in S_n$, then $\wt{\AS}_\om(Y,Z) = \langle
\wt{\AS}(Y,Z),\om \rangle$, where
\[
\wt{\AS}(Y,Z) := A_1(y_1)\cdots A_{n-1}(y_{n-1})
\wt{A}_{n-1}(z_{n-1})\cdots \wt{A}_1(z_1)
\]
and, for each variable $t$, the factors $A_i(t)$ and $\wt{A}_i(t)$ are
defined as in \S \ref{spsdd}.

Choose any integer $m\geq 0$, let
$\delta^\vee_{n-1}(m):=(m+1,\ldots,m+n-1)$, and
$\delta^\vee_{n-1}:=(1,2,\ldots,n-1)$. Let $\Omega:=(\omega_1,\ldots,
\omega_m)$ be an $m$-tuple of independent variables and
$\Omega+Y=(\omega_1,\ldots,\omega_m,y_1,y_2,\ldots)$ denote the
concatenation of the alphabets $\Omega$ and $Y$. For every $i\in
[1,m+n-1]$, define $\A_i(t):=(1+tu_{m+n-1})\cdots (1+tu_i)$ and
$\wt{\A}_i(t):=\A_i(t)^{-1}$.  Furthermore, set
$\A(\Omega):=\A_1(\omega_1)\cdots \A_m(\omega_m)$ and
\[
\wt{\AS}(\Omega+Y,Z) := \A(\Omega)
\A_{m+1}(y_1)\cdots \A_{m+n-1}(y_{n-1})
\wt{\A}_{m+n-1}(z_{n-1})\cdots \wt{\A}_{m+1}(z_1).
\]
Let $\om_0$ denote the longest element of $S_n$ and consider the polynomial
\begin{equation}
\label{defwteq}
\wt{\AS}_{1_m\times \om_0}(\Omega+Y,Z) := \langle 
\wt{\AS}(\Omega+Y,Z),1_m\times\om_0 \rangle.
\end{equation}

\begin{prop}
\label{SchubOm}
For every integer $m\geq 0$, we have 
\begin{equation}
\label{Stildeqom}
\wt{\AS}_{1_m\times \om_0}(\Omega+Y,Z) =
S^{(\delta^\vee_{n-1}(m),\delta^\vee_{n-1})}_{\delta_{n-1}}(h(\Omega+Y,-Z)),
\end{equation}
where the superscript $(\delta^\vee_{n-1}(m),\delta^\vee_{n-1})$
indicates the number of $\omega$, $y$, and $z$ variables used in each 
row of the flagged Schur polynomial.
\end{prop}
\begin{proof}
Let $\Omega':=(\omega'_1,\ldots,\omega'_m)$. Then
\[
\AS_{1_m\times \om_0}(\Omega+Y,\Omega'+Z) = 
\sum_{v_2^{-1} u v_1=\om_0} \AS_{v_2}(-Z)G_u(\Omega/\Omega')\AS_{v_1}(Y),
\]
so setting $\omega'_j=0$ for each $j\in [1,m]$ gives
\begin{equation}
\label{E1}
\AS_{1_m\times \om_0}(\Omega+Y,Z) =
\sum_{u,v_1\, :\, \ell(u)+\ell(v_1) = \ell(uv_1)} G_u(\Omega)\AS_{v_1}(Y)\AS_{uv_1\om_0}(-Z).
\end{equation}
Moreover, it follows from equations (\ref{keyid}) and (\ref{defwteq})
that
\begin{equation}
\label{E2}
\wt{\AS}_{1_m\times \om_0}(\Omega+Y,Z) = 
\sum_{u,v_1\, :\, \ell(u)+\ell(v_1) = \ell(uv_1)}G_u(\Omega)\AS_{v_1}(Y)\AS_{\om_0 uv_1}(-Z).
\end{equation}

Suppose that $Y_1,\ldots,Y_\ell$ and $Z_1\ldots,Z_\ell$ denote finite sets of
independent variables. For any integer vector $\al=(\al_1,\ldots,\al_\ell)$,
we introduce the {\em multi-Schur polynomial}
\[
S_{\al}(Y_1-Z_1\,;\,\cdots\,;\, Y_\ell-Z_\ell)
:=\det(h_{\al_i+j-i}(Y_i/Z_i))_{1\leq i,j\leq \ell}.
\]
These generalize the supersymmetric Schur polynomials defined in \S
\ref{examps}.

A permutation $\om$ is called {\em vexillary} if it is
$2143$-avoiding, that is, there is no sequence $i<j<k<r$ such that
$\om_j<\om_i<\om_r < \om_k$.  The Schubert polynomials which are
indexed by vexillary permutations may be expressed as multi-Schur
polynomials (see \cite[(6.16)]{M2}).  Let $Y_{(r)}:=(y_1,\ldots,y_r)$
and $Z_{(r)}:=(z_1,\ldots,z_r)$ for each $r\geq 0$.  Since the
permutation $1_m\times \om_0$ is vexillary, we deduce from
loc.\ cit.\ that
\[
\AS_{1_m\times \om_0}(\Omega+Y,\Omega'+Z) = 
S_{\delta_{n-1}}(\Omega+Y_{(1)}-\Omega'-Z_{(n-1)} \,;\, \cdots \,;
\, \Omega+Y_{(n-1)}-\Omega'-Z_{(1)})
\]
and therefore, setting $\omega'_j=0$ for each $j\in [1,m]$, that
\begin{equation}
\label{E3}
\AS_{1_m\times \om_0}(\Omega+Y,Z) = 
S_{\delta_{n-1}}(\Omega+Y_{(1)}-Z_{(n-1)}\,;\,\cdots \,;\, \Omega+Y_{(n-1)}-Z_{(1)}).
\end{equation}

Consider the duality isomorphism with respect to the $Z$-variables
$$D_z:\Z[\Omega,Y_{(n)},Z_{(n)}]\to \Z[\Omega,Y_{(n)},Z_{(n)}]$$ which
sends $(z_1,\ldots,z_n)$ to $(-z_n,\ldots,-z_1)$ and leaves all the
remaining variables fixed. Let $I^z_n\subset
\Z[\Omega,Y_{(n)},Z_{(n)}]$ be the ideal generated by the elementary
symmetric polynomials $e_i(Z_{(n)})$ for $1\leq i\leq n$.  It follows
from (\ref{E1}), (\ref{E2}), and Lemma \ref{Dcomp} that
\[
D_z(\AS_{1_m\times\om_0}(\Omega+Y,Z)+I_n^z) = \wt{\AS}_{1_m\times
  \om_0}(\Omega+Y,Z) + I^z_n.
\] 
Furthermore, for $r,s\in [1,n-1]$, equation (\ref{Deq}) implies that
\[
D_z(h_j(Y_{(r)}-Z_{(s)})) \equiv h_j(Y_{(r)},-Z_{(n-s)})
\]
modulo the ideal $I^z_n$. It follows from (\ref{E3}) that
\begin{align*}
D_z(\AS_{1_m\times \om_0}(\Omega+Y,Z)) & \equiv
S_{\delta_{n-1}}(h(\Omega+Y_1,-Z_1),\ldots,h(\Omega+Y_{n-1},-Z_{n-1})) \\
&= S^{(\delta^\vee_{n-1}(m),\delta^\vee_{n-1})}_{\delta_{n-1}}(h(\Omega+Y,-Z)).
\end{align*}
We deduce that equation (\ref{Stildeqom}) holds modulo $I^z_n$.

Let $\cH_n^z$ be the $\Z[\Omega,Y_{(n)}]$-linear subspace of
$\Z[\Omega,Y_{(n)},Z_{(n)}]$ spanned by the monomials $z^\al$ for
$0\leq \al\leq \delta_{n-1}$.  Then the monomial expression
(\ref{monomialsum}) for the flagged Schur polynomial
$S^{(\delta^\vee_{n-1}(m),\delta^\vee_{n-1})}_{\delta_{n-1}}(h(\Omega+Y,-Z))$
in (\ref{Stildeqom}) proves that the latter lies in $\cH_n^z$. Since
the same is clearly true of $\wt{\AS}_{1_m\times\om_0}(\Omega+Y,Z)$,
the proposition follows.
\end{proof}

The next result is obtained by setting $m=0$ in Proposition
\ref{SchubOm}.

\begin{cor}
\label{Schubom0}
We have $\dis 
\wt{\AS}_{\om_0}(Y,Z) = S^{(\delta^\vee_{n-1},\delta^\vee_{n-1})}_{\delta_{n-1}}(h(Y,-Z))$.
\end{cor}

\section{The type C theory}
\label{tCt}

\subsection{Schubert polynomials and divided differences}
\label{spsdds}
The Weyl group for the root system of type $\text{B}_n$ or
$\text{C}_n$ is the {\em hyperoctahedral group} $W_n$, which consists
of signed permutations on the set $\{1,\ldots,n\}$. The group $W_n$ is
generated by the simple transpositions $s_i=(i,i+1)$ for $1\leq i \leq
n-1$ and the sign change $s_0(1)=\ov{1}$ (as is customary, we use a
bar to denote an entry with a negative sign). There is a natural
embedding $W_n\hookrightarrow W_{n+1}$ defined by adding the fixed
point $n+1$, and we let $W_\infty :=\cup_n W_n$.  The notions of
length, reduced words, and descents of elements of $W_\infty$ are
defined as in the case of the symmetric group $S_\infty$, only now the
simple reflections are indexed by the integers in the set $\N_0
:=\{0,1,\ldots\}$.

The {\em nilCoxeter algebra} $\cW_n$ of the hyperoctahedral group
$W_n$ is the free associative algebra with unit generated by the
elements $u_0,u_1,\ldots,u_{n-1}$ modulo the relations
\[
\begin{array}{rclr}
u_i^2 & = & 0 & i\in \N_0\ ; \\
u_iu_j & = & u_ju_i & |i-j|\geq 2\ ; \\
u_iu_{i+1}u_i & = & u_{i+1}u_iu_{i+1} & i>0\ ; \\
u_0u_1u_0u_1 & = & u_1u_0u_1u_0.
\end{array}
\]
For any $w\in W_n$, define $u_w := u_{a_1}\ldots u_{a_\ell}$, where
$a_1\cdots a_\ell$ is any reduced word for $w$. Then the $u_w$ for
$w\in W_n$ form a free $\Z$-basis of $\cW_n$. As in \S \ref{spsdd}, we
denote the coefficient of $u_w\in \cW_n$ in the expansion of the
element $\xi\in \cW_n$ by $\langle \xi,w\rangle$.

Let $t$ be a variable and define
\begin{gather*}
C(t) := (1+t u_{n-1})\cdots(1+t u_1)(1+tu_0)
(1+t u_0)(1+t u_1)\cdots (1+t u_{n-1}).
\end{gather*}
Suppose that $X=(x_1,x_2,\ldots)$ is another infinite sequence
of commuting variables, let 
$C(X):=C(x_1)C(x_2)\cdots$, and for $w\in W_n$, define
\begin{equation}
\label{dbleC}
\CS_w(X\,;Y,Z) := \left\langle 
\tilde{A}_{n-1}(z_{n-1})\cdots \tilde{A}_1(z_1)C(X) A_1(y_1)\cdots 
A_{n-1}(y_{n-1}), w\right\rangle.
\end{equation}
Set $\CS_w(X\,;Y):=\CS_w(X\,;Y,0)$. The polynomials $\CS_w(X\,;Y)$ are
the type C Schubert polynomials of Billey and Haiman \cite{BH} and the
$\CS_w(X\,;Y,Z)$ are their double versions introduced by Ikeda,
Mihalcea, and Naruse \cite{IMN1}.

Note that $\CS_w$ is really a polynomial in the $Y$ and $Z$ variables,
with coefficients which are formal power series in $X$, with integer
coefficients.  These power series are symmetric in the $X$ variables,
since $C(s)C(t) = C(t)C(s)$, for any two commuting variables $s$ and
$t$ (see \cite[Prop.\ 4.2]{FK}). We set
\begin{equation}
\label{Fdef}
F_w(X):=\CS_w(X\,;0,0) = \left\langle C(X), w\right\rangle
\end{equation}
 and call $F_w$ the {\em type C Stanley symmetric function} indexed by
 $w\in W_n$. We deduce from (\ref{Fdef}) that the coefficients of
 $F_w(X)$ have a combinatorial interpretation (compare with
 \cite[(3.5)]{BH}). Observe also that we have $F_w=F_{w^{-1}}$.

The above definition of $\CS_w(X\,;Y,Z)$ implies that it is stable 
under the natural inclusions $W_n\hookrightarrow W_{n+1}$ of the 
Weyl groups, and hence is well defined for $w\in W_\infty$. Equation
(\ref{dbleC}) implies the relation
\begin{equation}
\label{dbleC2}
\CS_w(X\,;Y,Z) = \sum_{uv\om=w}\AS_{u^{-1}}(-Z)F_v(X)\AS_{\om}(Y)
\end{equation}
summed over all reduced factorizations $uv\om=w$ with $u,\om\in S_\infty$. 

The type C Stanley symmetric functions $F_w(X)$ lie in the ring $\Gamma$
of Schur $Q$-functions $Q_\la(X)$. For each $r\in \Z$, 
define the basic function $q_r(X)$ by the equation
\[
\prod_{i=1}^{\infty}\frac{1+x_it}{1-x_it} = \sum_{r=0}^{\infty}q_r(X)t^r.
\]
For any integer vector $\al$, let $q_\al:=\prod_iq_{\al_i}$, and
define $Q_\al:=Q_\al(X)$ by
\begin{equation}
\label{Qdef}
Q_\al:=\RR\, q_\al
\end{equation}
where the raising operator expression $\RR$ is given by
\[
\RR:=\prod_{i<j}\frac{1-R_{ij}}{1+R_{ij}}\,.
\]
Following \cite[III.8]{M1}, we can equivalently write formula
(\ref{Qdef}) using a Schur Pfaffian \cite{Sc}.  More precisely, for
integer vectors $\al=(\al_1,\al_2)$ with only two parts, we have
\[
Q_{(\al_1,\al_2)}=\frac{1-R_{12}}{1+R_{12}}\, q_{(\al_1,\al_2)} = 
q_{\al_1}q_{\al_2}+2\,\sum_{j\geq 1} (-1)^j q_{\al_1+j}q_{\al_2-j},
\]
while for an integer vector $\al=(\al_1,\ldots,\al_\ell)$ with three
or more components,
\begin{equation}
\label{giamQ}
Q_\al = \Pf(Q_{(\al_i,\al_j)})_{1\leq i<j \leq 2\ell'}
\end{equation}
where $\ell'$ is the least positive integer such that $2\ell' \geq
\ell$.

A partition $\la$ is strict if all its (non-zero) parts $\la_i$ 
are distinct. It is known that the $Q_\la(X)$ for $\la$ a strict partition 
form a free $\Z$-basis of the ring $\Gamma:=\Z[q_1(X),q_2(X),\ldots]$.
For any $w\in W_\infty$, we have an identity
\begin{equation}
\label{CStan0}
F_w(X) = \sum_{\la\, :\, |\la| = \ell(w)} e^w_{\la}\,Q_{\la}(X),
\end{equation}
summed over all strict partitions $\la$ with $|\la|=\ell(w)$. At this
stage we only need to know that an equation (\ref{CStan0}) exists with
$e^w_\la\in \Q$. This latter fact follows immediately from the
cancellation rule $C(t)C(-t)=1$ and a result of Pragacz
\cite[Thm.\ 2.11]{P} (see also \cite[\S 4]{FK}).  We refer to
\cite{B, BH, La} for three different combinatorial proofs that the
coefficients $e^w_{\la}$, when non-zero, are positive
integers. Equations (\ref{dbleC2}) and (\ref{CStan0}) show that the
Schubert polynomials $\CS_w(X\,;Y,Z)$ lie in the ring $\Gamma[Y,Z]$.

We define an action of $W_\infty$ on $\Gamma[Y,Z]$ by ring
automorphisms as follows. The simple reflections $s_i$ for $i>0$ act
by interchanging $y_i$ and $y_{i+1}$ and leaving all the remaining
variables fixed, as in \S \ref{spsdd}. The reflection $s_0$ maps $y_1$
to $-y_1$, fixes the $y_j$ for $j\geq 2$ and all the $z_j$, and
satisfies
\[
s_0(q_r(X)) := q_r(y_1,x_1,x_2,\ldots) =
q_r(X)+2\sum_{j=1}^ry_1^jq_{r-j}(X).
\]
For each $i\geq 0$, define the {\em divided difference operator}
$\partial_i^y$ on $\Gamma[Y,Z]$ by
\[
\partial_0^yf := \frac{f-s_0f}{-2y_1}, \qquad
\partial_i^yf := \frac{f-s_if}{y_i-y_{i+1}} \ \ \ \text{for $i>0$}.
\]
Consider the ring involution $\omega:\Gamma[Y,Z]\to\Gamma[Y,Z]$
determined by
\[
\omega(y_j) = -z_j, \qquad
\omega(z_j) = -y_j, \qquad
\omega(q_r(X))=q_r(X)
\]
and set $\partial_i^z:=\omega\partial_i^y\omega$ for each $i\geq 0$.

The polynomials $\CS_w(X\,;Y,Z)$ for $w\in W_{\infty}$ are the unique
family of elements of $\Gamma[Y,Z]$ satisfying the equations
\begin{equation}
\label{ddCeq}
\partial_i^y\CS_w = \begin{cases}
\CS_{ws_i} & \text{if $\ell(ws_i)<\ell(w)$}, \\ 
0 & \text{otherwise},
\end{cases}
\quad
\partial_i^z\CS_w = \begin{cases}
\CS_{s_iw} & \text{if $\ell(s_iw)<\ell(w)$}, \\ 
0 & \text{otherwise},
\end{cases}
\end{equation}
for all $i\geq 0$, together with the condition that the constant term
of $\CS_w$ is $1$ if $w=1$, and $0$ otherwise. As a consequence, the
the descents of $w$ and $w^{-1}$ determine the symmetries of the
polynomial $\CS_w(X\,;Y,Z)$. Note however the special role that
descents at {\em zero} play here. Furthermore, as in type A, one can
show that the double Schubert polynomials $\CS_w(X\,;Y,Z)$ represent
the universal Schubert classes in type C flag bundles, and therefore
degeneracy loci of symplectic vector bundles, in the sense of
\cite{Fu2}. For a simple proof of these assertions, see \cite[\S
  7.3]{T6}.  The {\em geometrization} of the Schubert polynomials
$\CS_w$ will be discussed in \S \ref{geometriz}.

\subsection{The Schubert polynomial indexed by the longest element}
\label{1stsec}

Let $w_0$ denote the longest element in $W_n$.  A formula for the top
single Schubert polynomial $\CS_{w_0}(X\,;Y)$ was given by Billey and
Haiman \cite[Prop.\ 4.15]{BH}.  In this section, we derive the
analogue of their result for the double Schubert polynomial
$\CS_{w_0}(X\, ; Y,Z)$, and use it to give a new proof of a Pfaffian
formula for $\CS_{w_0}$ due to Ikeda, Mihalcea, and Naruse
\cite[Thm.\ 1.2]{IMN1}.

Observe first that $w_0u=uw_0$ for any $u\in S_n$. Using this and
equation (\ref{dbleC2}) gives
\[
\CS_{w_0}(X\, ; Y,Z) = \sum_{wvu=w_0} F_w(X)\AS_{u^{-1}}(-Z)\AS_v(Y) \\
\]
where the $u,v$ in the sum lie in $S_n$, and the factorization
$wvu=w_0$ is reduced. It follows that
\begin{equation}
\label{ffeq}
\CS_{w_0}(X\, ; Y,Z) = 
\sum_{\sigma\in S_n}F_{w_0\sigma^{-1}}(X)\wt{\AS}_\sigma(Y,Z).
\end{equation}

\begin{prop}
\label{Ctopfirst}
The equation 
\begin{equation}
\label{Ctopeqfirst}
\CS_{w_0}(X\, ;Y,Z) = \sum_{\la\subset \delta_{n-1}}  Q_{\delta_n+\la}(X) 
S^{(\delta^\vee_{n-1},\delta^\vee_{n-1})}_{\delta_{n-1}/\la'}(h(Y,-Z))
\end{equation}
holds in $\Gamma[Y,Z]$.
\end{prop}
\begin{proof}
Recall that $\om_0$ denotes the longest permutation in $S_n$.
According to \cite[Thm.\ 3.16]{BH} and \cite[Thm.\ 3.15]{La}, we have,
for every $u\in S_n$,
\begin{equation}
\label{FQeq}
F_{w_0u}(X) = \sum_\la a_{\la}^{u^{-1}\om_0} Q_{\delta_n+\la}(X),
\end{equation}
summed over partitions $\la$.  We deduce from (\ref{ffeq}) and
(\ref{FQeq}) that
\[
\CS_{w_0}(X\, ; Y,Z) = \sum_\la Q_{\delta_n+\la}(X) \sum_{\sigma\in
  S_n} a_{\la}^{\sigma\om_0} \wt{\AS}_\sigma(Y,Z).
\]
It therefore suffices to show that for every partition $\la$, we have 
\begin{equation}
\label{BHgen}
\sum_{\sigma\in S_n} a_{\la}^{\sigma\om_0} \wt{\AS}_\sigma(Y,Z) =
S^{(\delta^\vee_{n-1},\delta^\vee_{n-1})}_{\delta_{n-1}/\la'}(h(Y,-Z)).
\end{equation}
This is a generalization of \cite[Eqn.\ (4.63)]{BH}, and its proof is 
similar. 

Choose any integer $m\geq 0$ and define $\Omega:=(\omega_1,\ldots,
\omega_m)$ and $\delta^\vee_{n-1}(m):=(m+1,\ldots,m+n-1)$ as in 
\S \ref{AStil}. For any integer vector $\gamma=(\gamma_1,\ldots,
\gamma_{n-1})$, we have 
\begin{equation}
\label{hdeq}
h^{(\delta^\vee_{n-1}(m),\delta^\vee_{n-1})}_{\gamma}(\Omega+Y,-Z) 
= \sum_{\al\geq 0} h_{\gamma-\al}(\Omega)
h_{\al}^{(\delta^\vee_{n-1},\delta^\vee_{n-1})}(Y,-Z)
\end{equation}
summed over all compositions $\al=(\al_1,\ldots,
\al_{n-1})$. Moreover, for any such composition $\al$, we have
$s_{\delta_{n-1}-\al}(\Omega)=0$ unless $\delta_{n-1}-\al+\delta_{n-2}
= \s(\la+\delta_{n-2})$ for some partition $\la$ and permutation $\s
\in S_{n-1}$, in which case $s_{\delta_{n-1}-\al}(\Omega)=(-1)^{\s}
s_\la(\Omega)$.  Using this and equation (\ref{hdeq}), we compute that
\begin{gather*}
S^{(\delta^\vee_{n-1}(m),\delta^\vee_{n-1})}_{\delta_{n-1}}(h(\Omega+Y,-Z)) 
= \prod_{i<j}(1-R_{ij}) \, h^{(\delta^\vee_{n-1}(m),\delta^\vee_{n-1})}_{\delta_{n-1}}
(\Omega+Y,Z) \\
= \sum_{\al\geq 0}
s_{\delta_{n-1}-\al}(\Omega)\,h^{(\delta^\vee_{n-1},\delta^\vee_{n-1})}_{\al}(Y,-Z) \\
= \sum_{\la}s_\la(\Omega)\sum_{\s\in S_{n-1}} (-1)^{\s}\,
h^{(\delta^\vee_{n-1},\delta^\vee_{n-1})}_{\delta_{n-1}+\delta_{n-2}-\s(\la+\delta_{n-2})}(Y,-Z).
\end{gather*}
We deduce that 
\begin{equation}
\label{dedeq}
S^{(\delta^\vee_{n-1}(m),\delta^\vee_{n-1})}_{\delta_{n-1}}(\Omega+Y,-Z) =
\sum_{\la\subset \delta_{n-1}}s_\la(\Omega)\,
S^{(\delta^\vee_{n-1},\delta^\vee_{n-1})}_{\delta_{n-1}/\la}(h(Y,-Z)).
\end{equation}

On the other hand, it follows from (\ref{keyid}), (\ref{Geq}), and the
definition (\ref{defwteq}) that
\[
\wt{\AS}_{1_m\times \om_0}(\Omega+Y,Z) = 
\sum_{uv = \om_0}G_u(\Omega)\wt{\AS}_v(Y,Z)
= \sum_\la \sum_{v\in S_n} a_\la^{\om_0v^{-1}} s_\la(\Omega) \wt{\AS}_v(Y,Z)
\]
and hence
\begin{equation}
\label{sedeq}
\wt{\AS}_{1_m\times \om_0}(\Omega+Y,Z) = 
\sum_\la s_\la(\Omega) \sum_{\sigma\in S_n} a_{\la'}^{\sigma\om_0}\wt{\AS}_\sigma(Y,Z).
\end{equation}

By Proposition \ref{SchubOm}, the left hand sides of equations 
(\ref{dedeq}) and (\ref{sedeq}) coincide. Comparing the coefficients 
of $s_\la(\Omega)$ on the right hand sides of the same equations
completes the proof of (\ref{BHgen}), and hence of the proposition.
\end{proof}

Our next goal is to express the top polynomial $\CS_{w_0}(X\, ; Y,Z)$
as a multi-Schur Pfaffian analogous to equations (\ref{Qdef}) and
(\ref{giamQ}).  For any $k,r\in \Z$, we define the polynomial
${}^kc^r_p={}^kc^r_p(X\, ; Y,Z)$ by
\begin{equation}
\label{ceq}
{}^kc^r_p:=\sum_{i=0}^p\sum_{j=0}^p q_{p-j-i}(X)h^{-k}_i(Y)h_j^r(-Z).
\end{equation}
The polynomials ${}^kc^r_p$ were first studied by Wilson in
\cite[Def.\ 6 and Prop.\ 6]{W}.  For any integer sequences
$\al,\be,\rho$, define ${}^\rho c^\be_\al:=\prod_i
     {}^{\rho_i}c^{\be_i}_{\al_i}$.  Given any raising operator $R$,
     let $R\, {}^\rho c^\be_{\al} := {}^\rho c^\be_{R\al}$.  Finally,
     define the {\em multi-Schur Pfaffian} ${}^\rho Q^\be_\al(c)$ by
\[
{}^\rho Q^\be_\al(c):= \RR\, {}^\rho c^\be_\al.
\]

\begin{prop}
\label{Ctop}
The equation 
\begin{equation}
\label{Ctopeq}
\CS_{w_0}(X\, ;Y,Z) = {}^{\delta_{n-1}}Q_{\delta_n+\delta_{n-1}}^{-\delta_{n-1}}(c)
\end{equation}
holds in $\Gamma[Y,Z]$.
\end{prop}
\begin{proof}
If we set ${}^kh^r_m(Y,-Z):=
\sum_{j=0}^m h^{-k}_j(Y)h_{m-j}^r(-Z)$, then we have
\[
{}^kc^r_p= \sum_{j=0}^p q_{p-j}(X) \,{}^kh_j^r(Y,-Z).
\]
In particular, if $k,r\geq 0$, then we have
\begin{equation}
\label{heq}
{}^kh^{-r}_m(Y,-Z)=e_m(y_1,\ldots,y_k,-z_1,\ldots,-z_r) = 
e_m(Y_{(k)}, -Z_{(r)})
\end{equation}
so that ${}^kh^{-r}_m(Y,-Z)=0$ whenever $m>k+r$.

For any integer sequences $\al,\be,\rho$, define ${}^\rho
h^\be_\al:=\prod_i {}^{\rho_i}h^{\be_i}_{\al_i}$.  Notice, using
(\ref{heq}), that for any integer vector
$\gamma=(\gamma_1,\ldots,\gamma_n)$, we have
\[
{}^{\delta_{n-1}}c_{\gamma}^{-\delta_{n-1}} = 
\sum_{0\leq \al \leq 2\delta_{n-1}}q_{\gamma-\al}(X)\,
{}^{\delta_{n-1}}h_\al^{-\delta_{n-1}}(Y,-Z),
\]
and hence, by the definition (\ref{Qdef}),  
\[
{}^{\delta_{n-1}}Q_{\delta_n+\delta_{n-1}}^{-\delta_{n-1}}(c) =
\sum_{0\leq \al \leq 2\delta_{n-1}}Q_{\delta_n+\delta_{n-1}-\al}(X)\,
{}^{\delta_{n-1}}h_\al^{-\delta_{n-1}}(Y,-Z).
\]
Recall for example from \cite[Lemma 1.3]{BKT2} that the Schur
$Q$-functions $Q_\gamma(X)$ are alternating in the components
$(\gamma_i,\gamma_j)$ of the index $\gamma$, provided that
$\gamma_i+\gamma_j>0$. Therefore, we have
$Q_{\delta_n+\delta_{n-1}-\al}(X)=0$ in the above sum unless
\[
\delta_n+\delta_{n-1}-\al = 1^n+\s(\delta_{n-1}+\la)
\]
for some partition $\la \subset \delta_{n-1}$ and permutation $\s\in S_{n-1}$.
Observe that $\al = 2 \delta_{n-1} - \s(\delta_{n-1}+\la)$ is uniquely 
determined from $\la$ and $\s$. It follows that
\begin{gather*}
{}^{\delta_{n-1}}Q_{\delta_n+\delta_{n-1}}^{-\delta_{n-1}}(c) = 
\sum_{\la\subset\delta_{n-1}}\sum_{\s\in S_{n-1}} Q_{1^n+\s(\delta_{n-1}+\la)}(X)\,
{}^{\delta_{n-1}}h^{-\delta_{n-1}}_{2 \delta_{n-1} - \s(\delta_{n-1}+\la)}(Y,-Z) \\
= \sum_{\la\subset\delta_{n-1}} Q_{\delta_n+\la}(X)\sum_{\s\in S_{n-1}}(-1)^\s\,
{}^{\delta_{n-1}}h^{-\delta_{n-1}}_{2 \delta_{n-1} - \s(\la+\delta_{n-1})}(Y,-Z) \\
= \sum_{\la\subset\delta_{n-1}} Q_{\delta_n+\la}(X) \,
S^{(\delta_{n-1},\delta_{n-1})}_{\delta_{n-1}/\la}(e(Y,-Z)).
\end{gather*}
Taking $\la=\delta_{n-1}$, $\mu=\la'$, $k=n$, $\ell = n-1$, and 
$t=(y_1,-z_1,\ldots,y_{n-1}, -z_{n-1})$ in Lemma \ref{goodlem} gives
\begin{equation}
\label{Sdual}
S^{(\delta_{n-1},\delta_{n-1})}_{\delta_{n-1}/\la}(e(Y,-Z)) =
S^{(\delta^\vee_{n-1},\delta^\vee_{n-1})}_{\delta_{n-1}/\la'}(h(Y,-Z)).
\end{equation}
The result now follows by using equations (\ref{Ctopeqfirst}) and
(\ref{Sdual}).
\end{proof}

It is easy to show (see for example \cite[\S 8.2]{IM}) that equation
(\ref{Ctopeq}) is equivalent to the Pfaffian formula for $\CS_{w_0}$
found in \cite[Thm.\ 1.2]{IMN1}.

\subsection{The Schubert polynomials indexed by maximal elements}
\label{ttotheta}

Consider a sequence $\fraka \ :\ a_1<\cdots < a_p$ of nonnegative
integers with $a_p<n$. The sequence $\fraka$ parametrizes a parabolic
subgroup $W_\fraka$ of $W_n$, generated by the simple reflections
$s_i$ for $i\notin\{a_1,\ldots, a_p\}$. We let $W_n^\fraka$ denote the
set of minimal length left $W_\fraka$-coset representatives. Recall
that
\[
W_n^\fraka = \{w\in W_n\ |\ \ell(ws_i) = \ell(w)+1,\  \forall\, i \notin
\{a_1,\ldots, a_p\}\}.
\]
Let $w_0(\fraka)$ denote the longest element in $W_n^\fraka$; we have
\[
w_0(\fraka)=\begin{cases}
\ov{a_2}\cdots\ov{1}\ov{a_3}\cdots \ov{a_2+1}\cdots\ov{n}\cdots\ov{a_p+1}
& \text{if $a_1=0$}, \\
1\cdots a_1\ov{a_2}\cdots\ov{a_1+1}\cdots\ov{n}\cdots\ov{a_p+1}
& \text{if $a_1>0$}.
\end{cases}
\]
It is known (see for example \cite[\S 2]{St}) that $W_n^\fraka$ is an order 
ideal of the left weak Bruhat order of $W_n$, and that $w_0(\fraka)$
is the unique maximal element of $W_n^\fraka$ under this ordering.

Fix an integer $k$ with $0\leq k < n$. The elements of the set
$W_n^{(k)}$ are the {\em $k$-Grassmannian} elements of $\wt{W}_n$.
Let $w^{(k,n)}=1\cdots k\, \ov{n}\cdots \ov{k+1}$ denote the longest
element of $W_n^{(k)}$. Following \cite[\S 6.2]{TW}, we will require a
formula analogous to (\ref{Ctopeq}) for the Schubert polynomial
$\CS_{w^{(k,n)}}(X\,;Y,Z)$, which maps to Kazarian's multi-Schur Pfaffian
formula from \cite[Thm.\ 1.1]{Ka}. Similar Pfaffian formulas for
the Schubert polynomials $\CS_{w_0(\fraka)}(X\,; Y,Z)$ were obtained
by Anderson and Fulton in \cite{AF1}. Ikeda and Matsumura \cite[\S
  8.2]{IM} gave proofs of these formulas by applying left difference
operators to the top Schubert polynomial $\CS_{w_0}$, and we will
follow that approach here.

For every $i\geq 0$, the operator $\partial_i:=\partial_i^z$ on
$\Gamma[Y,Z]$ satisfies the Leibnitz rule
\begin{equation}
\label{LeibR}
\partial_i(fg) = (\partial_if)g+(s_if)\partial_ig.
\end{equation}
The argument depends on the following basic lemmas.

\begin{lemma}[\cite{IM}, Lemma 5.4]
\label{ddlem}
Suppose that $p,r\in \Z$ and let $k\geq 0$. For all $i\geq 0$, we have 
\[
\partial_i ({}^kc_p^r)= 
\begin{cases}
{}^kc_{p-1}^{r+1} & \text{if $r=\pm i$}, \\
0 & \text{otherwise}.
\end{cases}
\]
\end{lemma}

\begin{lemma}[\cite{IM}, Lemma 8.2]
\label{imlem1}
Suppose that $i\geq 0$ and $k>0$. Then we have
\[
{}^kc_p^{-i} = {}^{k-1}c_p^{-i-1}+(z_{i+1}+y_k)\, {}^{k-1}c_{p-1}^{-i}.
\]
\end{lemma}

\begin{lemma}[\cite{IM}, Prop.\ 5.4]
\label{imlem2}
Suppose that $k,r\geq 0$ and $p> k+r$. Then we have
\[
{}^{(k_1,\ldots,k,k,\ldots,k_\ell)}
Q_{(p_1,\ldots,p,p,\ldots,p_\ell)}^{(i_1,\ldots,-r,-r,\ldots,i_\ell)} (c)= 0.
\]
\end{lemma}

\begin{example}
Let $\delta^*_n:=(n,n-1,\ldots,2)\in \Z^{n-1}$.
For any integer sequence $\al=(\al_1,\ldots,\al_n)$, we have
\[
\partial_0({}^{n-1}c^{1-n}_{\al_1} \cdots {}^1c^{-1}_{\al_{n-1}}
{}^0 c^0_{\al_n}) = {}^{n-1}c^{1-n}_{\al_1} \cdots {}^1c^{-1}_{\al_{n-1}}
{}^0c^1_{\al_n-1}.
\]
It follows from this and equations (\ref{ddCeq}) and (\ref{Ctopeq})
that
\[
\CS_{s_0w_0}(X\,;Y,Z) = \partial_0 \, \CS_{w_0}(X\,;Y,Z)  = 
{}^{\delta_{n-1}}Q_{\delta^*_n+\delta_{n-1}}^{-\delta_{n-1}}(c).
\]
Arguing as in \S \ref{1stsec}, we can show that
\begin{equation}
\label{s0w0eq}
\CS_{s_0w_0}(X\,;Y,Z) =
\sum_{\la\subset\delta^*_n} Q_{\delta_{n-1}+\la}(X) 
S^{(\delta_{n-1},\delta_{n-1})}_{\delta^*_n/\la}(e(Y,-Z)).
\end{equation}
\end{example}

\begin{prop}
\label{wknprop}
We have
\begin{equation}
\label{wkneq}
\CS_{w^{(k,n)}}(X\,;Y,Z)={}^{(k,k,\ldots, k)}Q_{(n+k,n+k-1,\ldots, 2k+1)}^{(1-n,2-n,\ldots,-k)}(c).
\end{equation}
in $\Gamma[Y,Z]$.
\end{prop}
\begin{proof}
If $v^{(k,n)}=\ov{k}\cdots \ov{1} \,\ov{n}\cdots \ov{k+1}$ is the
longest element in $W_n^{(0,k)}$, then we have a reduced factorization
$w_0= v_1v_2 v^{(k,n)}$, where
\begin{equation}
\label{eqv1}
v_1 := (s_{k-1} \cdots s_1)(s_{k-1} \cdots s_2)\cdots(s_{k-1}s_{k-2})s_{k-1}
\end{equation}
if $k\geq 2$, and $v_1:=1$, otherwise, while
\begin{equation}
\label{eqv2}
v_2 := (s_{n-1} \cdots s_{k+1})(s_{n-1} \cdots s_{k+2})
\cdots(s_{n-1}s_{n-2})s_{n-1}.
\end{equation}
Using (\ref{ddCeq}), this implies the equation
\[
\CS_{v^{(k,n)}} = \partial_{n-1} (\partial_{n-2}\partial_{n-1}) \cdots
(\partial_{k+1} \cdots \partial_{n-1})
\cdot\partial_{k-1}(\partial_{k-2}\partial_{k-1}) \cdots (\partial_1
\cdots \partial_{k-1}) \CS_{w_0}.
\]

Assume that $k\geq 2$, as the proof when $k\in \{0,1\}$ is easier.
Using Lemmas \ref{ddlem} and \ref{imlem1}, for any $p\in \Z$
we have 
\begin{equation}
\label{k1eq}
\partial_{k-1} {}^{k-1}c^{1-k}_p = {}^{k-1}c^{2-k}_{p-1}  = 
{}^{k-2}c^{1-k}_{p-1}+(z_{k-1}+y_{k-1}){}^{k-2}c^{2-k}_{p-2}.
\end{equation}
Let $\epsilon_j$ denote the $j$-th standard basis vector in 
$\Z^n$. The Leibnitz rule and (\ref{k1eq}) imply that for any
integer vector $\al=(\al_1,\ldots,\al_n)$, we have
\[
\partial_{k-1} {}^{\delta_{n-1}}c^{-\delta_{n-1}}_{\alpha} = 
{}^{\delta_{n-1}-\epsilon_{n+1-k}}c^{-\delta_{n-1}}_{\al-\epsilon_{n+1-k}}+
(z_{k-1}+y_{k-1}){}^{\delta_{n-1}-\epsilon_{n+1-k}}
c^{-\delta_{n-1}+\epsilon_{n+1-k}}_{\al - 2\epsilon_{n+1-k}}.
\]
We deduce from this and Lemma \ref{imlem2} that 
\begin{align*}
\partial_{k-1}
{}^{\delta_{n-1}}Q_{\delta_n+\delta_{n-1}}^{-\delta_{n-1}}(c) &= 
{}^{\delta_{n-1}-\epsilon_{n+1-k}}Q^{-\delta_{n-1}}_{\delta_n+\delta_{n-1}-\epsilon_{n+1-k}}(c) \\ 
&= {}^{(n-1,\ldots, k, k-2,k-2,\ldots, 1)}
Q_{(2n-1,\ldots,2k+1, 2k-2, 2k-3,\ldots,1)}^{(1-n,\ldots,-1,0)}(c).
\end{align*}
Iterating this calculation gives
\[
(\partial_1\cdots \partial_{k-1})\CS_{w_0} = 
{}^{(n-1,\ldots, k, k-2,k-3,\ldots, 0)}Q_{(2n-1,\ldots, 2k+1, 2k-2, 2k-4,\ldots,2,1)}^{(1-n,\ldots,-1,0)}(c)
\]
and furthermore
\[
\partial_{k-1}(\partial_{k-2}\partial_{k-1}) \cdots (\partial_1
\cdots \partial_{k-1}) \CS_{w_0} = 
{}^{(n-1,\ldots, k, 0,0,\ldots, 0)}Q_{(2n-1,\ldots, 2k+1, k, k-1,\ldots,1)}^{(1-n,\ldots,-1,0)}(c).
\]
Applying the operator $\partial_{n-1} (\partial_{n-2}\partial_{n-1}) \cdots
(\partial_{k+1} \cdots \partial_{n-1})$ to both sides of the above equation, 
we similarly obtain
\[
\CS_{v^{(k,n)}} =
{}^{(k,k,\ldots, k, 0,0,\ldots, 0)}Q_{(n+k,n+k-1,\ldots, 2k+1, k, k-1,\ldots,1)}^{(1-n,\ldots,-1,0)}(c).
\]

Since $v^{(k,n)} = (s_0\cdots s_{k-1}) \cdots (s_0s_1)s_0 w^{(k,n)}$, 
equation (\ref{ddCeq}) gives
\[
\CS_{w^{(k,n)}} = \partial_0 (\partial_1\partial_0) \cdots (\partial_{k-1}
\cdots\partial_0)\CS_{v^{(k,n)}}.
\]
Finally, since
${}^{(\rho,r)}Q^{(\be,b)}_{(\al,0)}(c) = {}^{\rho}Q^{\be}_{\al}(c)$
for any integers $r$ and $b$, it follows that
\begin{align*}
\CS_{w^{(k,n)}} &=  \partial_0 (\partial_1\partial_0) \cdots (\partial_{k-1}
\cdots\partial_0)\CS_{v^{(k,n)}} \\
&= \partial_0 (\partial_1\partial_0) \cdots (\partial_{k-2}
\cdots\partial_0){}^{(k,\ldots, k, 0,\ldots, 0)}Q_{(n+k,\ldots, 2k+1, k-1, \ldots,1)}
^{(1-n,\ldots,-k,2-k\ldots,0)}(c) \\
&=
{}^{(k,\ldots, k)}Q_{(n+k,\ldots, 2k+1)}^{(1-n,\ldots,-k)}(c).\\
\end{align*}
\end{proof}

Since $w^{(k,n)}=(w^{(k,n)})^{-1}$, the polynomial
$\CS_{w^{(k,n)}}(X\,;Y,Z)$ is symmetric in the $Z$ variables as
well as in the $Y$ variables, however this is not reflected in
equation (\ref{wkneq}). The next proposition makes this symmetry
apparent. Recall (for example from \cite[\S 6.2]{Fu3}) that for any
three partitions $\la,\mu$, and $\nu$, the {\em Littlewood-Richardson
  number} $N^\la_{\mu\nu}$ is the nonnegative integer defined by the
equation of Schur $S$-functions
\[
s_\mu(t)s_\nu(t) = \sum_\la N^\la_{\mu\nu} s_\la(t).
\]
Let $\mu_0:=(2k)^{n-k} = (2k,\ldots,2k)$, and for every $\mu\subset
\mu_0$, define $\mu^\vee:=(2k-\mu_{n-k},\ldots,2k-\mu_1)$. Note that
$(n+k,n+k-1,\ldots, 2k+1) = \delta_{n-k}+\mu_0$.

\begin{prop}
\label{wknprop2}
We have
\begin{align*}
\label{wkneq2}
\CS_{w^{(k,n)}}(X\,;Y,Z)&={}^{(k,\ldots, k)}Q_{\delta_{n-k}+\mu_0}^{(-k,\ldots,-k)}(c) \\
& = \sum_{\nu_1,\nu_2\subset\mu\subset \mu_0}N_{\nu_1\nu_2}^{\mu} 
Q_{\delta_{n-k}+\mu^\vee}(X)s_{\nu_1'}(Y_{(k)})s_{\nu_2'}(-Z_{(k)}),
\end{align*}
in $\Gamma[Y,Z]$, where $N_{\nu_1\nu_2}^\mu$ denotes a
Littlewood-Richardson number.
\end{prop}
\begin{proof}
For any integer vector $\gamma=(\gamma_1,\ldots,\gamma_{n-k})$, we have
\begin{equation}
\label{kceq}
{}^{(k,\ldots, k)}c_{\gamma}^{(1-n,\ldots,-k)} = \sum_{0\leq\al\leq
  \delta_{n-k-1}+\mu_0}q_{\gamma-\al}(X)e^{(n+k-1,\ldots,2k)}_\al(Y_{(k)},-Z)
\end{equation}
where the alphabet $Y_{(k)}$ in the factor
$e^{(n+k-1,\ldots,2k)}_\al(Y_{(k)},-Z)$ is constant, while $Z$ varies
down from $Z_{(n-1)}$ to $Z_{(k)}$. Now using (\ref{wkneq}) and
(\ref{kceq}) while applying the raising operator $\RR$, along with the
alternating property of Schur $Q$-functions, gives
\begin{align*}
\CS_{w^{(k,n)}}(X\,;Y,Z) &= 
\sum_{0\leq \al\leq \delta_{n-k-1}+\mu_0}
Q_{\delta_{n-k}+\mu_0-\al}(X)e^{(n+k-1,\ldots,2k)}_\al(Y_{(k)},-Z) \\
&= \sum_{\mu\subset \mu_0}Q_{\delta_{n-k}+\mu}(X) 
\det(e_{2k-\mu_j+j-i}^{(n+k-1,\ldots,2k)}(Y_{(k)},-Z))_{1\leq i,j\leq n-k} \\
&= \sum_{\mu\subset \mu_0}Q_{\delta_{n-k}+\mu}(X)
S_{\mu_0/\mu}^{(n+k-1,\ldots,2k)}(e(Y_{(k)},-Z)).
\end{align*}

We claim that for each partition $\mu\subset \mu_0$, we have
\begin{align*}
\det(e_{2k-\mu_j+j-i}^{(n+k-1,\ldots,2k)}(Y_{(k)},-Z)) &=
\det(e_{2k-\mu_j+j-i}^{(2k,\ldots,2k)}(Y_{(k)},-Z_{(k)})) \\
&= S_{\mu_0/\mu}(e(Y_{(k)},-Z_{(k)})).
\end{align*}
The proof of this follows \cite[(3.4)]{M2}. For each $i,j$ 
with $1\leq i,j \leq n-k$,
\begin{align*}
e^{n+k-i}_{2k-\mu_j-i+j}(Y_{(k)},-Z) &= 
e_{2k-\mu_j-i+j}(Y_{(k)},-Z_{n-i}) \\
&= \sum_{p=1}^{n-k}e_{p-i}(-B_i)
e_{2k-\mu_j+j-p}(Y_{(k)},-Z_{(k)}),
\end{align*}
where $B_i=(z_{k+1},\ldots,z_{n-i})$ (in particular, $B_{n-k}=\emptyset$). 
Therefore the matrix
\[
\{e^{n+k-i}_{2k-\mu_j-i+j}(Y_{(k)},-Z)\}_{1\leq i,j\leq n-k}
\]
is the product of the matrix
\[
\{e_{p-i}(-B_i)\}_{1\leq i,p\leq n-k},
\]
which is unitriangular, and the matrix
\[
\{e_{2k-\mu_j+j-p}(Y_{(k)},-Z_{(k)})\}_{1\leq p,j\leq n-k}.
\]
Taking determinants completes the proof of the claim. 

Since
$S_{\mu_0/\mu}(e(Y_{(k)},-Z_{(k)}))=S_{\mu^\vee}(e(Y_{(k)},-Z_{(k)}))$,
we deduce that
\begin{align*}
\CS_{w^{(k,n)}}(X\,;Y,Z) &= \sum_{\mu\subset \mu_0}Q_{\delta_{n-k}+\mu}(X)
S_{\mu^\vee}(e(Y_{(k)},-Z_{(k)})) \\
&= {}^{(k,\ldots, k)}Q_{\delta_{n-k}+\mu_0}^{(-k,\ldots,-k)}(c).
\end{align*}
Furthermore, using \cite[I.(5.9)]{M1}, we compute that
\begin{align*}
\CS_{w^{(k,n)}}(X\,;Y,Z)
&= \sum_{\mu\subset \mu_0}Q_{\delta_{n-k}+\mu^\vee}(X)
S_{\mu}(e(Y_{(k)},-Z_{(k)})) \\
&= \sum_{\mu\subset \mu_0}Q_{\delta_{n-k}+\mu^\vee}(X)
s_{\mu'}(Y_{(k)},-Z_{(k)}) \\
&= \sum_{\nu\subset\mu\subset \mu_0}Q_{\delta_{n-k}+\mu^\vee}(X)
s_{\mu'/\nu'}(Y_{(k)})s_{\nu'}(-Z_{(k)}) \\
&= \sum_{\nu_1,\nu_2\subset\mu\subset \mu_0}N_{\nu_1'\nu_2'}^{\mu'} 
Q_{\delta_{n-k}+\mu^\vee}(X)s_{\nu_1'}(Y_{(k)})s_{\nu_2'}(-Z_{(k)}).
\end{align*}
Since we have $N_{\nu_1'\nu_2'}^{\mu'}=N_{\nu_1\nu_2}^\mu$, the result follows.
\end{proof}

Although we only require the formula for $\CS_{w^{(k,n)}}(X\,;Y,Z)$, we 
will record the general result here for comparison with the orthogonal
case, which is discussed in \S \ref{ttoeta}.
According to \cite{AF1} and \cite[Thm.\ 8.2]{IM}, we have
\[
\CS_{w_0(\fraka)}(X\,; Y,Z) = 
{}^{\rho(\fraka)}Q^{\be(\fraka)}_{\la(\fraka)}(c),
\]
where $\la(\fraka)$, $\be(\fraka)$, and $\rho(\fraka)$ denote the sequences
\[
\la(\fraka)=(n+a_p,\ldots,2a_p+1,\ldots,
a_i+a_{i+1},\ldots,2a_i+1,\ldots,a_1+a_2,\ldots,2a_1+1)\,;
\]
\[
\be(\fraka)=(1-n,\ldots,-a_p,\ldots,
1-a_{i+1},\ldots,-a_i,\ldots,1-a_2,\ldots,-a_1)\,;
\]
and
\[
\rho(\fraka)=(a_p^{n-a_p},\ldots,a_i^{a_{i+1}-a_i},\ldots,a_1^{a_2-a_1}).
\]

\subsection{Theta polynomials}
\label{tps}

The next step in this program is to prove formulas for the Schubert
polynomials indexed by the $k$-Grassmannian elements of $W_\infty$. We
will see that they may be expressed using {\em theta polynomials}.

We say that a partition $\la$ is {\em $k$-strict} if no part greater
than $k$ is repeated, that is, $\la_j>k$ implies $\la_{j+1}<\la_j$ for
each $j\geq 1$.  There is an explicit bijection between
$k$-Grassmannian elements $w$ of $W_\infty$ and $k$-strict partitions
$\la$, such that the elements in $W_n$ correspond to those partitions
whose diagram fits inside an $(n-k)\times (n+k)$ rectangle.  According
to \cite[\S 4.1 and \S 4.4]{BKT1}, if the element $w$ corresponds to the
$k$-strict partition $\la$, then the bijection is given by the
equations
\[
\la_i=\begin{cases} 
|w_{k+i}|+k & \text{if $w_{k+i}<0$}, \\
\#\{p\leq k\, :\, w_p> w_{k+i}\} & \text{if $w_{k+i}>0$}.
\end{cases}
\]
Using the above bijection, we attach to any $k$-strict partition $\la$
a finite set of pairs 
\begin{equation}
\label{Cweq}
\cC(\la) := \{ (i,j)\in \N\times\N \ |\ 1\leq i<j \ \ \text{and} \ \ 
 w_{k+i}+ w_{k+j} < 0 \}
\end{equation}
and a sequence $\beta(\la)=\{\beta_j(\la)\}_{j\geq 1}$ defined by
\begin{equation}
\label{css}
\beta_j(\la):=\begin{cases}
w_{k+j}+1 & \text{if $w_{k+j}<0$}, \\
w_{k+j} & \text{if $w_{k+j}>0$}.
\end{cases}
\end{equation}

Following \cite{BKT2}, let $\la$ be any $k$-strict partition, and
consider the raising operator expression $R^\la$ given by
\begin{equation}
\label{Req}
R^{\la} := \prod_{i<j}(1-R_{ij})\prod_{(i,j)\in\cC(\la)}
(1+R_{ij})^{-1}.
\end{equation}
For any integer sequences $\al$ and $\be$, define $c_{\al}^\be:=
\prod_i {}^kc_{\al_i}^{\be_i}$. According to \cite{TW, W}, the {\em
  double theta polynomial} $\Ti_\la(X\,; Y_{(k)},Z)$ is defined by
\begin{equation}
\label{thetadef}
\Ti_\la(X\,; Y_{(k)},Z) := R^\la\,c^{\be(\la)}_{\la}.
\end{equation}
The single theta polynomial $\Ti_\la(X\,;Y_{(k)})$ of \cite{BKT2} is given by
\[
\Ti_\la(X\,;Y_{(k)}):=\Ti_\la(X\,; Y_{(k)},0).
\] 
Note that we are working here with the images of the theta polynomials
$\Ti_\la(c)$ and $\Ti_\la(c\, |\, t)$ from \cite{BKT2, TW} in the ring
$\Gamma[Y,Z]$ of double Schubert polynomials, following \cite{T5, T6}.

Fix a rank $n$ and let $$\la_0:=(n+k,n+k-1,\ldots,2k+1)$$ be the
$k$-strict partition associated to the $k$-Grassmannian element 
$w^{(k,n)}$ of maximal length in $W_n$. We deduce from Proposition 
\ref{wknprop} and the definition (\ref{thetadef}) that 
\begin{equation}
\label{ceqt}
\CS_{w^{(k,n)}}(X\,; Y,Z) = \Ti_{\la_0}(X\,; Y_{(k)},Z)
\end{equation}
in $\Gamma[Y,Z]$.

It follows from \cite[Lemma 5.5]{IM} that for all $i\geq 1$ and indices 
$p$ and $q$, we have
\[
\partial_i\,c_{(p,q)}^{(-i,i)} = c_{(p-1,q)}^{(-i+1,i+1)} +
c_{(p,q-1)}^{(-i+1,i+1)} = (1+R_{12}) \, c_{(p-1,q)}^{(-i+1,i+1)}.
\]
Using this identity and Lemma \ref{ddlem}, it is shown in
\cite[Prop.\ 5]{TW} that if $\la$ and $\mu$ are $k$-strict partitions
such that $|\la|=|\mu|+1$ and $w_\la=s_iw_{\mu}$ for some simple
reflection $s_i\in W_\infty$, then we have
\begin{equation}
\label{pareq}
\partial_i \Ti_\la(X\,; Y_{(k)},Z)  = \Ti_{\mu}(X\,; Y_{(k)},Z)
\end{equation}
in $\Gamma[Y,Z]$.  Now (\ref{ceqt}) and (\ref{pareq}) imply that for
any $k$-strict partition $\la$ with associated $k$-Grassmannian
element $w_\la$, we have
\[
\CS_{w_\la}(X\,; Y,Z) = \Ti_{\la}(X\,; Y_{(k)},Z)
\]
in $\Gamma[Y,Z]$. In particular, we recover the equality 
\begin{equation}
\label{ceqtsing}
\CS_{w_\la}(X\,; Y) =  \Ti_{\la}(X\,; Y_{(k)})
\end{equation}
in $\Gamma[Y]$ from \cite[Prop.\ 6.2]{BKT2} for the single polynomials.

\subsection{Mixed Stanley functions and splitting formulas}
\label{msfsfsC}

Following \cite[\S 2]{T5}, for any $w\in W_\infty$, the 
{\em double mixed Stanley function} $J_w(X\,;Y/Z)$ is defined by
the equation
\[
J_w(X\,;Y/Z) := \langle \tilde{A}(Z)C(X)A(Y), w\rangle = 
\sum_{uv\om=w}G_{u^{-1}}(-Z)F_v(X)G_\om(Y),
\]
where the sum is over all reduced factorizations $uv\om=w$ with
$u,\om\in S_\infty$. The single mixed Stanley function $J_w(X\,;Y)$ is
given by setting $Z=0$ in $J_w(X\,;Y/Z)$. Observe that $J_w(X\,;Y/Z)$ is
separately symmetric in the three sets of variables $X$, $Y$, and $Z$, 
and that we have $J_w(X\,;0)=F_w(X)$.

Fix an integer $k\geq 0$.  We say that an element $w\in W_\infty$ is
{\em increasing up to $k$} if $0 < w_1 < w_2 < \cdots < w_k$ (this
condition is automatically true if $k=0$).  If $w$ is increasing up to
$k$, then \cite[Eqn.\ (2.5)]{BH} and equation (\ref{keyid}) have a
natural analogue for the {\em restricted mixed Stanley function}
$J_w(X\,;Y_{(k)})$, which is obtained from $J_w(X;Y)$ after setting
$y_i=0$ for $i>k$. In this case, according to \cite[Prop.\ 5]{T5}, we
have
\begin{equation}
\label{keyidC}
\CS_w(X\,;Y)= 
\sum_{v(1_k\times \om) = w}J_v(X\,;Y_{(k)})\AS_\om(y_{k+1},y_{k+2},\ldots),
\end{equation}
where the sum is over all reduced factorizations $v(1_k\times \om) =
w$ in $W_\infty$ with $\om\in S_\infty$. Moreover, there is a double
version of equation (\ref{keyidC}) which is parallel to
(\ref{keyid2}). Let $J_v(X\,;Y_{(k)}/Z_{(\ell)})$ denote the power
series obtained from $J_v(X\,;Y/Z)$ by setting $y_i=z_j=0$ for all
$i>k$ and $j>\ell$. Then if $w$ is increasing up to $k$ and $w^{-1}$
is increasing up to $\ell$, we have
\begin{equation}
\label{keyidC2}
\CS_w(X\,;Y,Z)= 
\sum \AS_{u^{-1}}(-Z_{>\ell}) J_v(X\,;Y_{(k)}/Z_{(\ell)})\AS_\om(Y_{>k}),
\end{equation}
where $Y_{>k}:=(y_{k+1},y_{k+2},\ldots)$,
$-Z_{>\ell}:=(-z_{\ell+1},-z_{\ell+2},\ldots)$, and the sum is over
all reduced factorizations $(1_{\ell}\times u)v(1_k\times \om) = w$ in
$W_\infty$ with $u,\om\in S_\infty$.

We say that an element $w\in W_\infty$ is {\em compatible} with
the sequence $\fraka \, :\, a_1 < \cdots < a_p$ of elements of $\N_0$
if all descent positions of $w$ are contained in $\fraka$.  Let
$\frakb \, :\, b_1 < \cdots <b_q$ be a second sequence of elements of 
$\N_0$ and assume that $w$ is compatible with $\fraka$ and
$w^{-1}$ is compatible with $\frakb$. We say that a reduced
factorization $u_1\cdots u_{p+q-1} = w$ is {\em compatible} with
$\fraka$, $\frakb$ if $u_i\in S_\infty$ for all $i\neq q$, 
$u_j(i)=i$ whenever $j<q$ and $i\leq b_{q-j}$ or
whenever $j>q$ and $i \leq a_{j-q}$.
Set $Y_i := \{y_{a_{i-1}+1},\ldots,y_{a_i}\}$ for each $i\geq 1$ and
$Z_j := \{z_{b_{j-1}+1},\ldots,z_{b_j}\}$ for each $j\geq 1$. 

\begin{prop}
\label{split2prop}
Suppose that $w$ and $w^{-1}$ are compatible with $\fraka$ and 
$\frakb$, respectively. Then
the Schubert polynomial $\CS_w(X\,;Y,Z)$ satisfies
\[
\CS_w = \sum
G_{u_1}(0/Z_q)\cdots G_{u_{q-1}}(0/Z_2)
J_{u_q}(X\,;Y_1/Z_1)G_{u_{q+1}}(Y_2) \cdots G_{u_{p+q-1}}(Y_p)
\]
summed over all reduced factorizations $u_1\cdots u_{p+q-1} = w$
compatible with $\fraka$, $\frakb$. 
\end{prop}
\begin{proof}
The result is established
by combining the identity (\ref{keyid}) with (\ref{keyidC2}).
\end{proof}

If $w$ is increasing up to $k$, then the following generalization of
equation (\ref{CStan0}) holds (see \cite[Thm.\ 1]{T5}):
\begin{equation}
\label{CmStan}
J_w(X\,;Y_{(k)}) = \sum_{\la\, :\, |\la| = \ell(w)} e^w_{\la}\,\Ti_{\la}(X;Y_{(k)}),
\end{equation}
where the sum is over $k$-strict partitions $\la$ with $|\la| =
\ell(w)$. The {\em mixed Stanley coefficients} $e^w_{\la}$ in
(\ref{CmStan}) are nonnegative integers. In fact, to any $w\in
W_\infty$ increasing up to $k$ we associate a {\em $k$-transition
  tree} $T^k(w)$ whose leaves are $k$-Grassmannian elements, and
$e^w_{\la}$ is equal to the number of leaves of the tree $T^k(w)$
which have shape $\la$. The proof of (\ref{CmStan}) in \cite{T5} is a
straightforward application of Billey's transition equations for
symplectic flag varieties \cite{B} combined with equation
(\ref{ceqtsing}).

Assume that $w$ is increasing up to $k$ and 
$w^{-1}$ is increasing up to $\ell$.
At present there is no clear analogue of equation (\ref{CmStan}) 
for the (restricted) double mixed Stanley function
$J_w(X\,;Y_{(k)}/Z_{(\ell)})$. However, we have
\begin{align}
\label{Jeq1}
J_w(X\,;Y_{(k)}/Z_{(\ell)}) 
&= \sum_{uv=w}G_{u^{-1}}(-Z_{(\ell)})J_v(X\,;Y_{(k)}) \\
\label{Jeq2}
&= \sum_{uv=w^{-1}}G_{u^{-1}}(Y_{(k)})J_v(X\,;-Z_{(\ell)}),
\end{align}
where the factorizations under the sum signs are reduced with $u\in
S_\infty$. We can now use equations (\ref{Geq}) and (\ref{CmStan}) in
(\ref{Jeq1}) and (\ref{Jeq2}) to obtain two dual expansions of
$J_w(X\,;Y_{(k)}/Z_{(\ell)})$ as a positive sum of products of
Schur $S$-polynomials with theta polynomials.

\begin{example}
\label{Jex}
Let $w=231=s_1s_2\in W_3$ and take $k=\ell=1$. We have
\[
\CS_{231}(X\,;Y,Z) = q_2(X)+q_1(X)(y_1+y_2-z_1) +(y_1-z_1)(y_2-z_1)
\]
and hence $J_{231}(X\,; Y_{(1)}/Z_{(1)}) =  q_2(X)+q_1(X)(y_1-z_1) -(y_1-z_1)z_1$.
Equality (\ref{Jeq1}) gives
\begin{align*}
J_{231}(X\,; Y_{(1)}/Z_{(1)})
&= J_{231}(X\,;Y_{(1)}) + G_{213}(-Z_{(1)})J_{132}(X\,;Y_{(1)}) + G_{312}(-Z_{(1)}) \\
&= \Ti_2(X\,;y_1) + s_1(-z_1)\Ti_1(X\,;y_1) + s_2(-z_1) \\
&= (q_2(X)+q_1(X)y_1) + (-z_1)(q_1(X)+y_1) + z_1^2,
\end{align*}
while equality (\ref{Jeq2}) gives
\begin{align*}
J_{231}(X\,; Y_{(1)}/Z_{(1)}) 
&= J_{312}(X\,;-Z_{(1)}) + G_{132}(Y_{(1)})J_{213}(X\,;-Z_{(1)})
+ G_{231}(Y_{(1)}) \\
&= \Ti_{(1,1)}(X\,;-z_1) + s_1(y_1)\Ti_1(X\,;-z_1) + s_{(1,1)}(y_1) \\
&= (q_2(X)-q_1(X)z_1+z_1^2) + y_1(q_1(X)-z_1).
\end{align*}
\end{example}

\begin{thm}[\cite{T5}, Cor.\ 1]
\label{Csplitthm}
Suppose that $w$ is compatible with $\fraka$ and $w^{-1}$ is compatible
with $\frakb$, where $b_1=0$. Then we have
\begin{equation}
\label{dbCSsplitting}
\CS_w = \sum_{\underline{\la}} 
f^w_{\underline{\la}}\,
s_{\la^1}(0/Z_q)\cdots s_{\la^{q-1}}(0/Z_2)
\Ti_{\la^q}(X\,;Y_1)s_{\la^{q+1}}(Y_2)\cdots s_{\la^{p+q-1}}(Y_p)
\end{equation}
summed over all sequences of partitions
$\underline{\la}=(\la^1,\ldots,\la^{p+q-1})$ with $\la^q$ $a_1$-strict,
where 
\begin{equation}
\label{dbfdef0}
f^w_{\underline{\la}} := \sum_{u_1\cdots u_{p+q-1} = w}
a_{\la^1}^{u_1}\cdots a_{\la^{q-1}}^{u_{q-1}}
e_{\la^q}^{u_q}a_{\la^{q+1}}^{u_{q+1}}\cdots a_{\la^{p+q-1}}^{u_{p+q-1}}
\end{equation}
summed over all reduced factorizations $u_1\cdots u_{p+q-1} = w$
compatible with $\fraka$, $\frakb$.
\end{thm}
\begin{proof}
The result follows from Proposition \ref{split2prop} by using 
equations (\ref{Geq}) and (\ref{CmStan}).
\end{proof}

Proposition \ref{split2prop} and Theorem \ref{Csplitthm} are 
symplectic analogues of Proposition \ref{split1A} and Theorem
\ref{Asplitthm}, and similar remarks about their algebraic,
combinatorial, and geometric significance apply. We refer the reader
to \cite{T5}, \cite[\S 4 and \S 6]{T6}, and \S \ref{geometriz} of the
present paper for further details and for examples which illustrate
computations of the mixed Stanley coefficients $e^w_\la$.

\section{The type D theory}
\label{tDt}

For the orthogonal Lie types B and D we work with coefficients in the
ring $\Z[\frac{1}{2}]$. For $w\in W_\infty$, the type B double
Schubert polynomial $\BS_w$ of \cite{IMN1} is related to the type C
Schubert polynomial by the equation $\BS_w=2^{-s(w)}\CS_w$, where
$s(w)$ denotes the number of indices $i$ such that $w_i<0$. We will
therefore omit any further discussion of type B, and concentrate on
the even orthogonal case.  The exposition is parallel to that of \S
\ref{tCt}, but there are some interesting variations in the results
and in their proofs.

\subsection{Schubert polynomials and divided differences}
The Weyl group $\wt{W}_n$ for the root system $\text{D}_n$ is the
subgroup of $W_n$ consisting of all signed permutations with an even
number of sign changes.  The group $\wt{W}_n$ is an extension of $S_n$
by the element $s_\Box=s_0s_1s_0$, which acts on the right by
\[
(w_1,w_2,\ldots,w_n)s_\Box=(\ov{w}_2,\ov{w}_1,w_3,\ldots,w_n).
\]
There is a natural embedding $\wt{W}_n\hookrightarrow \wt{W}_{n+1}$ of
Weyl groups defined by adjoining the fixed point $n+1$, and we let
$\wt{W}_\infty := \cup_n \wt{W}_n$. The elements of the set $\N_\Box
:=\{\Box,1,\ldots\}$ index the simple reflections in $\wt{W}_\infty$;
these are used to define the reduced words and descents of elements in
$\wt{W}_\infty$ as in the previous sections.

The nilCoxeter algebra $\wt{\cW}_n$
of $\wt{W}_n$ is the free associative algebra with unit generated by
the elements $u_\Box,u_1,\ldots,u_{n-1}$ modulo the relations
\[
\begin{array}{rclr}
u_i^2 & = & 0 & i\in \N_\Box\ ; \\
u_\Box u_1 & = & u_1 u_\Box \\
u_\Box u_2 u_\Box & = & u_2 u_\Box u_2 \\
u_iu_{i+1}u_i & = & u_{i+1}u_iu_{i+1} & i>0\ ; \\
u_iu_j & = & u_ju_i & j> i+1, \ \text{and} \ (i,j) \neq (\Box,2).
\end{array}
\]

As in \S \ref{spsdds}, for any $w\in \wt{W}_n$, choose a reduced word
$a_1\cdots a_\ell$ for $w$, and define $u_w := u_{a_1}\ldots u_{a_\ell}$.
Denote the coefficient of $u_w\in \wt{\cW}_n$ in the expansion of
the element $\xi\in \wt{\cW}_n$ in the $u_w$ basis by $\langle
\xi,w\rangle$.  Let $t$ be a variable and, following Lam \cite{La},
define
\[
D(t) := (1+t u_{n-1})\cdots (1+t u_2)(1+t u_1)(1+t u_\Box)
(1+t u_2)\cdots (1+t u_{n-1}).
\]
Let $D(X):=D(x_1)D(x_2)\cdots$, and for $w\in \wt{W}_n$, define
\begin{equation}
\label{dbleD}
\DS_w(X\,;Y,Z) := \left\langle 
\tilde{A}_{n-1}(z_{n-1})\cdots \tilde{A}_1(z_1) D(X) A_1(y_1)\cdots 
A_{n-1}(y_{n-1}), w\right\rangle.
\end{equation}
The power series $\DS_w(X\,;Y):=\DS_w(X\,;Y,0)$ are the type D
Billey-Haiman Schubert polynomials, and the $\DS_w(X\,;Y,Z)$ are their
double versions from \cite{IMN1}.

The double Schubert polynomial $\DS_w(X\,;Y,Z)$ is stable under the
natural inclusions $\wt{W}_n\hookrightarrow \wt{W}_{n+1}$, and hence
is well defined for $w\in \wt{W}_\infty$.  We set
$$E_w(X):=\DS_w(X\,;0,0) = \left\langle D(X), w\right\rangle$$ and
call $E_w$ the {\em type D Stanley symmetric function} indexed by
$w\in \wt{W}_n$. Observe that we have $E_w=E_{w^{-1}}$. Equation
(\ref{dbleD}) implies the relation
\begin{equation}
\label{dbleD2}
\DS_w(X\,;Y,Z) = \sum_{uv\om=w}\AS_{u^{-1}}(-Z)E_v(X)\AS_{\om}(Y)
\end{equation}
summed over all reduced factorizations $uv\om=w$ with $u,\om\in S_\infty$.

For each strict partition $\la$, the Schur $P$-function $P_\la(X)$ is
defined by the equation $P_\la(X):= 2^{-\ell(\la)}Q_\la(X)$, where
$\ell(\la)$ denotes the length of $\la$.  The type D Stanley symmetric
functions $E_w(X)$ lie in the ring $\Gamma':=\Z[P_1,P_2,\ldots]$ of
Schur $P$-functions. In fact, for any $w\in \wt{W}_\infty$, 
we have an equation
\begin{equation}
\label{DStan0}
E_w(X) = \sum_{\la\, :\, |\la| = \ell(w)} d^w_{\la}\,P_{\la}(X)
\end{equation}
summed over all strict partitions $\la$ with $|\la|=\ell(w)$. Since
$D(t)D(-t)=1$, it follows from \cite[Thm.\ 2.11]{P} that an identity
(\ref{DStan0}) exists with coefficients $d^w_\la\in \Z$. Given
equation (\ref{dbleD2}), this implies that $\DS_w(X\,;Y,Z)$ is an
element of $\Gamma'[Y,Z]$, for any $w\in \wt{W}_\infty$. For three
different proofs that $d^w_{\la}\geq 0$, see \cite{B, BH, La}.

We define an action of $\wt{W}_\infty$ on $\Gamma'[Y,Z]$ by ring
automorphisms as follows. The simple reflections $s_i$ for $i>0$ act
by interchanging $y_i$ and $y_{i+1}$ and leaving all the remaining
variables fixed, as in \S \ref{spsdd}. The reflection $s_\Box$ maps 
$(y_1,y_2)$ to $(-y_2,-y_1)$, fixes the $y_j$ for $j\geq 3$ and all 
the $z_j$, and satisfies, for any $r\geq 1$,
\begin{align*}
\label{s_Box}
s_\Box(P_r(X)) &:= P_r(y_1,y_2,x_1,x_2,\ldots)  \\
&= P_r(X)+(y_1+y_2)\sum_{j=0}^{r-1}\left(\sum_{a+b=j}y_1^ay_2^b\right)
Q_{r-1-j}(X).
\end{align*}
For each $i\in \N_\Box$, define the divided difference operator
$\partial_i^y$ on $\Gamma'[Y,Z]$ by
\[
\partial_\Box^yf := \frac{f-s_\Box f}{-y_1-y_2}, \qquad
\partial_i^yf := \frac{f-s_if}{y_i-y_{i+1}} \ \ \ \text{for $i>0$}.
\]
Consider the ring involution $\omega:\Gamma'[Y,Z]\to\Gamma'[Y,Z]$
determined by
\[
\omega(y_j) = -z_j, \qquad
\omega(z_j) = -y_j, \qquad
\omega(P_r(X))=P_r(X)
\]
and set $\partial_i^z:=\omega\partial_i^y\omega$ for each $i\in \N_\Box$.

The polynomials $\DS_w(X\,;Y,Z)$ for $w\in \wt{W}_{\infty}$ are the unique
family of elements of $\Gamma'[Y,Z]$ satisfying the equations
\begin{equation}
\label{Dddeq}
\partial_i^y\DS_w = \begin{cases}
\DS_{ws_i} & \text{if $\ell(ws_i)<\ell(w)$}, \\ 
0 & \text{otherwise},
\end{cases}
\quad
\partial_i^z\DS_w = \begin{cases}
\DS_{s_iw} & \text{if $\ell(s_iw)<\ell(w)$}, \\ 
0 & \text{otherwise},
\end{cases}
\end{equation}
for all $i\in \N_\Box$, together with the condition that the constant
term of $\DS_w$ is $1$ if $w=1$, and $0$ otherwise. As in \S
\ref{spsdd} and \S \ref{spsdds}, it follows that descents of $w$ and
$w^{-1}$ determine the symmetries of the double Schubert polynomial
$\DS_w(X\,;Y,Z)$, and that the polynomials $\DS_w$ represent
degeneracy loci of even orthogonal vector bundles, in the sense of
\cite{Fu2}.

\subsection{Schur $P$-functions and their double analogues}
\label{spfan}

Let $n\geq 1$ be an integer and $\ell\in [1,n]$. Let 
$\al=(\al_1\ldots,\al_\ell)$ be a composition, and define 
a polynomial $P^{(\ell)}_\al(x_1,\ldots,x_n)$ by the equation
\begin{equation}
\label{Peq1}
P^{(\ell)}_\al(x_1,\ldots,x_n) := \frac{1}{(n-\ell)!}\sum_{\om\in S_n}
\om\left(x_1^{\al_1}\cdots x_\ell^{\al_\ell}\prod_{i\leq \ell,\, i<j\leq n}
\frac{x_i+x_j}{x_i-x_j}\right).
\end{equation}
Let $S_{n-\ell}$ denote the subgroup of 
$S_n$ consisting of permutations of $\{\ell+1,\ldots,n\}$. Since the
expression $\dis x^\al\prod_{i\leq \ell, i<j\leq n}\frac{x_i+x_j}{x_i-x_j}$ 
is symmetric in $(x_{\ell+1},\ldots,x_n)$, we deduce that
\begin{equation}
\label{Peqn}
P^{(\ell)}_\al(x_1,\ldots,x_n) = 
\sum_{\sigma\in S_n/S_{ n-\ell}} \sigma\left(x_1^{\al_1}\cdots x_\ell^{\al_\ell}
\prod_{i\leq \ell, i<j\leq n}\frac{x_i+x_j}{x_i-x_j}\right).
\end{equation}
It follows from \cite[Prop.\ 1.1(c)]{Iv1} that
$P^{(\ell)}_\al(x_1,\ldots,x_n)=0$ if $\al_i=\al_j$ for some $i\neq j$.
Hence, the polynomial $P^{(\ell)}_\al(x_1,\ldots,x_n)$ is alternating in
the indices $(\al_1,\ldots, \al_\ell)$.

\begin{lemma}
\label{newlemma}
Assume that $n$ is even. If $\al_\ell=0$, then we have
\[
P^{(\ell)}_\al(x_1,\ldots,x_n)=\begin{cases}
0 & \text{if $\ell$ is odd}, \\ P_\al^{(\ell-1)}(x_1,\ldots,x_n) &
\text{if $\ell$ is even}.
\end{cases}
\]
\end{lemma}
\begin{proof}
According to \cite[Prop.\ 2.4]{Iv2}, for any $m\geq 1$, we have 
\begin{equation}
\label{ivlem}
\sum_{\om\in S_m} \om\left(\prod_{j=2}^m\frac{x_1+x_j}{x_1-x_j}\right) =
\begin{cases} 0 & \text{if $m$ is even}, \\ (m-1)! & \text{if $m$ is odd}.
\end{cases}
\end{equation}

Let $H\cong S_{n+1-\ell}$ denote the subgroup of $S_n$ consisting of
  permutations of $\{\ell,\ldots,n\}$, and set $P_\al^{(\ell,n)}:=
P^{(\ell)}_\al(x_1,\ldots,x_n)$. Using (\ref{Peqn}) and equation 
(\ref{ivlem}), we compute that
\begin{align*}
(n-\ell)!P_\al^{(\ell,n)} &= \sum_{\sigma\in S_n/H}\sum_{\om\in H}
  \sigma \om \left(x^\al\prod_{i\leq \ell, i<j\leq
    n}\frac{x_i+x_j}{x_i-x_j}\right) \\ &= \sum_{\sigma\in S_n/H}
  \sum_{\om\in H} \sigma \left(x^\al\prod_{i< \ell, i<j\leq
    n}\frac{x_i+x_{\om(j)}}{x_i-x_{\om(j)}}\right)
  \prod_{j=\ell+1}^n\frac{x_{\sigma \om(\ell)}+x_{\sigma
      \om(j)}}{x_{\sigma \om(\ell)}-x_{\sigma \om(j)}} \\ &=
  \sum_{\sigma\in S_n/H} \sigma \left(x^\al\prod_{i< \ell, i<j\leq
    n}\frac{x_i+x_j}{x_i-x_j}\right) \cdot \sum_{\om\in H}
  \prod_{j=\ell+1}^n\frac{x_{\sigma \om(\ell)}+x_{\sigma
      \om(j)}}{x_{\sigma \om(\ell)}-x_{\sigma \om(j)}}
  \\ &= \begin{cases} 0 & \text{if $\ell$ is odd}, \\ (n-\ell)!\,
    P^{(\ell-1,n)}_\al & \text{if $\ell$ is even}.
\end{cases}
\end{align*}
\end{proof}

Let $t=(t_1,t_2\ldots)$ be a sequence of independent variables, as in
\S \ref{fsp}, and define $(x\,|\,t)^r:=(x-t_1)\cdots(x-t_r)$.  Given a
strict partition $\la$ of length $\ell$ and $n \geq \ell$, Ivanov's
double Schur $P$-function $P_\la(x_1,\ldots,x_n\, |\, t)$ is defined
by
\begin{equation}
\label{Peq2}
P_\la(x_1,\ldots,x_n\, |\, t):=\frac{1}{(n-\ell)!}\sum_{\om\in S_n}
\om\left(\prod_{i=1}^\ell(x_i\,|\, t)^{\la_i}\prod_{i\leq \ell,\, i<j\leq n}
\frac{x_i+x_j}{x_i-x_j}\right).
\end{equation}
Following \cite[\S 4.2]{IMN1}, we let $P_\la(X\,|\, t)$ denote the
(even) projective limit of the functions $P_\la(x_1,\ldots,x_{2m}\, |\,
t)$ as $m\to\infty$. We have that $P_\la(X\,|\, 0)=P_\la(X)$ is the
Schur $P$-function indexed by the partition $\la$.

\begin{prop}
\label{Pxtprop}
Let $\la$ be a strict partition of length $\ell$.

\medskip
\noin
{\em (a)} Suppose that $\ell$ is even and 
$\la=\delta_{\ell-1}+\mu$ for some partition $\mu$. Then 
\[
P_\la(X\, |\, t) = \sum_{\nu\subset \mu} P_{\delta_{\ell-1}+\nu}(X)S^\la_{\mu/\nu}(e(-t)).
\]

\medskip
\noin
{\em (b)} Suppose that $\ell$ is odd and $\la = 
\delta_{\ell}+\mu$ for some partition $\mu$. Then 
\[
P_\la(X\, |\, t) = \sum_{\nu\subset \mu} P_{\delta_{\ell}+\nu}(X)S^\la_{\mu/\nu}(e(-t)).
\]
\end{prop}
\begin{proof}
For any $r\geq 1$, we have $(x\,|\,t)^r=\sum_{p=0}^rx^pe^r_{r-p}(-t)$.
Therefore for any partition $\la$ of length $\ell$, we have
\begin{equation}
\label{Eeqn}
\prod_{i=1}^\ell(x_i\,|\, t)^{\la_i} = 
\sum_{\al\geq 0}x^{\al}e_{\la-\al}^{\la}(-t).
\end{equation}
It follows from (\ref{Peq1}), (\ref{Peq2}), and (\ref{Eeqn}) that 
for any $n\geq\ell$, we have
\[
P_\la(x_1,\ldots,x_n\, |\, t) =\sum_{0\leq \al\leq \la}
P_\al^{(\ell)}(x_1,\ldots,x_n)e_{\la-\al}^\la(-t).
\]
Taking the even projective limit as $n\to\infty$ gives
\begin{equation}
\label{epleq}
P_\la(X\, |\, t) =\sum_{0\leq \al\leq \la}P_\al^{(\ell)}(X)e_{\la-\al}^\la(-t).
\end{equation}

For part (a), using (\ref{epleq}), Lemma \ref{newlemma}, and the 
alternating property of the functions $P_\al^{(\ell)}(X)$ gives
\begin{gather*}
P_\la(X\, |\, t) = \sum_{\nu\subset \mu} \sum_{\s\in S_{\ell}}(-1)^\s
P_{\delta_{\ell-1}+\nu}(X)e^{\la}_{\delta_{\ell-1}+\mu-\s(\delta_{\ell-1}+\nu)}(-t) \\
= \sum_{\nu\subset \mu}P_{\delta_{\ell-1}+\nu}(X)S^\la_{\mu/\nu}(e(-t)).
\end{gather*}
For part (b), we similarly obtain 
\begin{gather*}
P_\la(X\, |\, t) = \sum_{\nu\subset \mu} \sum_{\s\in S_{\ell}}(-1)^\s
P_{\delta_{\ell}+\nu}(X)e^{\la}_{\delta_\ell+\mu-\s(\delta_{\ell}+\nu)}(-t) \\
= \sum_{\nu\subset \mu}P_{\delta_{\ell}+\nu}(X)S^\la_{\mu/\nu}(e(-t)).
\end{gather*}
\end{proof}

\begin{cor}
\label{Pcor}
\medskip
\noin
{\em (a)} If $n$ is odd, then 
\[
P_{2\delta_{n-1}}(X\, |\, t) = 
\sum_{\nu\subset \delta^*_n} P_{\delta_{n-2}+\nu}(X)S^{2\delta_{n-1}}_{\delta^*_n/\nu}(e(-t)).
\]

\medskip
\noin
{\em (b)} If $n$ is even, then 
\[
P_{2\delta_{n-1}}(X\, |\, t) = 
\sum_{\nu\subset \delta_{n-1}} P_{\delta_{n-1}+\nu}(X)S^{2\delta_{n-1}}_{\delta_{n-1}/\nu}(e(-t)).
\]
\end{cor}

According to \cite[\S 9]{Iv3} and \cite[\S 8.3]{IN}, we have a Pfaffian
formula 
\begin{equation}
\label{PPfaff}
P_\la(X\, |\,t) = \Pf(P_{\la_i,\la_j}(X\,|\, t))_{1\leq i< j \leq 2\ell'},
\end{equation}
where $2\ell'$ is the least even integer which is greater than or
equal to $\ell(\la)$. In equation (\ref{PPfaff}), we use the
conventions that $P_{a,b}(X\,|\,t) := -P_{b,a}(X\,|\,t)$ whenever
$0\leq a\leq b$, and $P_{a,0}(X\,|\, t):=P_a(X\,|\, t)$. We will
require a raising operator expression analogous to (\ref{Qdef}) for
the functions $P_\la(X\,|\, t)$.  This uses a more involved Pfaffian
formalism which stems from the work of Knuth \cite[\S 4]{Kn} and
Kazarian \cite[App.\ C and D]{Ka}.

For any $r\in \Z$, we define the polynomial
$\cc^r_p=\cc^r_p(X\,|\, t)$ by
\[
\cc^r_p:=\sum_{j=0}^p q_{p-j}(X)e_j^r(-t).
\]
For any integer sequences $\al$, $\be$, let 
\[
\wh{\cc}_\al^\be:= \wh{\cc}_{\al_1}^{\be_1}
\wh{\cc}_{\al_2}^{\be_2}\cdots
\]
where, for each $i\geq 1$, 
\[
\wh{\cc}_{\al_i}^{\be_i}:= \cc_{\al_i}^{\be_i} + 
\begin{cases}
(-1)^ie^{\al_i}_{\al_i}(-t) & 
\text{if $\be_i = \al_i > 0$}, \\
0 & \text{otherwise}.
\end{cases}
\]

If $R:=\prod_{i<j} R_{ij}^{n_{ij}}$ is any
raising operator, denote by $\supp(R)$ the set of all
indices $i$ and $j$ such that $n_{ij}>0$. Let $\al=(\al_1,\ldots,\al_\ell)$ 
and $\be=(\be_1,\ldots,\be_\ell)$
be integer vectors, set $\nu:=R\al$, and define
\[
R \star \wh{\cc}^\be_{\al} = \ov{\cc}^\be_{\nu} :=
\ov{\cc}_{\nu_1}^{\be_1}\cdots\ov{\cc}^{\be_\ell}_{\nu_\ell}
\]
where for each $i\geq 1$, 
\[
\ov{\cc}_{\nu_i}^{\be_i}:= 
\begin{cases}
\cc_{\nu_i}^{\be_i} & \text{if $i\in\supp(R)$}, \\
\wh{\cc}_{\nu_i}^{\be_i} & \text{otherwise}.
\end{cases}
\]

\begin{prop}
\label{Pprop}
For any strict partition $\la$, we have
\begin{equation}
\label{Pkk}
P_\la(X\, |\,t) = 2^{-\ell(\la)}\, \RR \star \wh{\cc}^{\la}_{\la}.
\end{equation}
\end{prop}
\begin{proof}
It follows from \cite{Kn, Ka} that the equation 
\begin{equation}
\label{Pkk2}
\RR \star \wh{\cc}^{\la}_{\la} = 2^{\ell(\la)}\Pf(P_{\la_i,\la_j}(X\,|\, t))_{i<j}
\end{equation}
holds if and only if it holds for all strict partitions $\la$ of length 
$\ell$ at most 3. The latter is a formal identity which is straightforward 
to check from the definitions; compare with \cite[App.\ A]{AF2} and 
\cite[\S 2.3]{IMN2}. We conclude from (\ref{PPfaff}) and (\ref{Pkk2}) that 
(\ref{Pkk}) is also true.
\end{proof}

\subsection{The Schubert polynomial indexed by the longest element}
\label{sofD}

Let $\wt{w}_0$ denote the longest element in $\wt{W}_n$. We have
\[
\wt{w}_0=\left\{ \begin{array}{cl}
           (\ov{1},\ldots,\ov{n}) & \mathrm{ if } \ n \ \mathrm{is} \ 
             \mathrm{even}, \\
           (1,\ov{2},\ldots,\ov{n}) & \mathrm{ if } \ n \ \mathrm{is} \ 
             \mathrm{odd}.
             \end{array} \right.
\]
A formula for the top single Schubert polynomial
$\DS_{\wt{w}_0}(X\,;Y)$ was given by Billey and Haiman
\cite[Prop.\ 4.15]{BH}.  In this section, we derive the analogue of
their result for the double Schubert polynomial $\DS_{\wt{w}_0}(X\, ;
Y,Z)$, and use it to give a combinatorial proof of the Pfaffian
formula for $\DS_{\wt{w}_0}$ from \cite[Thm.\ 1.2]{IMN1}.

\begin{prop}
\label{Dtopfirst}
If $n$ is even, then we have 
\[
\DS_{\wt{w}_0}(X\, ;Y,Z) = 
\sum_{\la\subset\delta_{n-1}} P_{\delta_{n-1}+\la}(X) 
S^{(\delta_{n-1},\delta_{n-1})}_{\delta_{n-1}/\la}(e(Y,-Z))
\]
in $\Gamma'[Y,Z]$.
\end{prop}
\begin{proof}
Using equation (\ref{dbleD2}), we have
\[
\DS_{\wt{w}_0}(X\, ; Y,Z) = \sum_{uwv=\wt{w}_0} E_w(X)\AS_{u^{-1}}(-Z)\AS_v(Y).
\]
summed over all reduced factorizations $uwv=\wt{w}_0$ in $\wt{W}_\infty$
with $u,v \in S_\infty$. If $n$ is even, then every permutation in $S_n$
commutes with $\wt{w}_0$, and it follows that
\begin{equation}
\label{ffeqD}
\DS_{\wt{w}_0}(X\, ; Y,Z) = 
\sum_{\sigma\in S_n}E_{\wt{w}_0\sigma^{-1}}(X)\wt{\AS}_\sigma(Y,Z).
\end{equation}

According to \cite[Thm.\ 3.16]{BH} and \cite[Thm.\ 5.14]{La}, we have,
for every $u\in S_n$,
\begin{equation}
\label{FPeq}
E_{\wt{w}_0u}(X) = \sum_\la a_{\la}^{u^{-1}\om_0} P_{\delta_{n-1}+\la}(X).
\end{equation}
We deduce from (\ref{ffeqD}) and (\ref{FPeq}) that
\[
\DS_{\wt{w}_0}(X\, ; Y,Z) = \sum_\la P_{\delta_{n-1}+\la}(X) \sum_{\sigma\in
  S_n} a_{\la}^{\sigma\om_0} \wt{\AS}_\sigma(Y,Z).
\]
The result now follows by combining equations (\ref{BHgen}) and (\ref{Sdual}).
\end{proof}

Fix $k\geq 0$, and set ${}^kc_p(X\, ; Y):= \sum_{i=0}^p
q_{p-i}(X)h^{-k}_i(Y)$, so that we have
\[
{}^kc^r_p(X\, ; Y,Z) = \sum_{j=0}^p {}^kc_{p-j}(X\, ; Y) h_j^r(-Z).
\]
Define ${}^kb_r := {}^kc_r$ for $r<k$, ${}^kb_r
:= \frac{1}{2}{}^kc_r$ for $r>k$, and set
\[
{}^kb_k := \frac{1}{2}{}^kc_k + \frac{1}{2}e^k_k(Y) \quad \text{and} \quad
{}^k\wt{b}_k := \frac{1}{2}{}^kc_k - \frac{1}{2}e^k_k(Y).
\]
Let $f_k$ be an indeterminate of degree $k$, which will be equal to
${}^kb_k$, ${}^k\wt{b}_k$, or $\frac{1}{2}\,{}^kc_k$,
depending on the context. We also let $f_0 \in\{0,1\}$. For any
$p,r\in \Z$, define ${}^k\wh{c}_p^r$ by
\[
{}^k\wh{c}_p^r:= {}^kc_p^r + 
\begin{cases}
(2f_k-{}^kc_k)e^{p-k}_{p-k}(-Z) & \text{if $r = k - p < 0$}, \\
0 & \text{otherwise}.
\end{cases}
\]
In particular, we have
\[
{}^0\wh{c}_p^r = {}^0c_p^r \pm
\begin{cases}
e^p_p(-Z) & \text{if $r = - p < 0$}, \\
0 & \text{otherwise},
\end{cases}
\]
and thus ${}^0\wh{c}_p^{-p} = {}^0c_p^{-p} \pm e^p_p(-Z)$ when $p > 0$, while 
${}^0\wh{c}_0^0 = 1$. It follows from \cite[Eq.\ (2.14)]{IMN2} that 
\[
P_r(X\, |\, Z)  = \frac{1}{2}({}^0c_r^{-r} - e^r_r(-Z)),
\]
that is, $P_r(X\,|\, Z)  = \frac{1}{2} \, {}^0\wh{c}_r^{-r}$ with the choice 
of $f_0 = 0$.

Recall that we have defined left divided differences
$\partial_i=\partial_i^z$ for each $i\in \N_\Box$. These operators
satisfy the same Leibnitz rule (\ref{LeibR}) as in the type C case.
We now have the following even orthogonal analogues of Lemmas
\ref{ddlem} and \ref{imlem1}.

\begin{lemma}[\cite{T7}, Prop.\ 2]
\label{Dddlem}
Suppose that $p,r\in \Z$ and let $k\geq 0$ and $i\geq 1$.

\medskip
\noin
{\em (a)} We have 
\[
\partial_i ({}^kc_p^r)= 
\begin{cases}
{}^kc_{p-1}^{r+1} & \text{if $r=\pm i$}, \\
0 & \text{otherwise}.
\end{cases}
\]
\medskip
\noin
{\em (b)} If $p>k$, we have 
\[
\partial_i({}^k\wh{c}_p^{k-p}) =
\begin{cases}
  {}^k\wh{c}_{p-1}^{k-p+1} & \text{if $i=p-k\geq 2$}, \\
  2f_k & \text{if $i=p-k=1$}, \\
0 & \text{otherwise}.
\end{cases}
\]
\end{lemma}

\begin{lemma}
\label{Dlem1}
Suppose that $i\geq 0$ and $k>0$. Then we have
\[
{}^k\wh{c}_p^{-i} = {}^{k-1}\wh{c}_p^{-i-1}+(z_{i+1}+y_k)\, {}^{k-1}\wh{c}_{p-1}^{-i}.
\]
\end{lemma}
\begin{proof}
We know from Lemma \ref{imlem1} that 
\begin{equation}
\label{imeq}
{}^kc_p^{-i} = {}^{k-1}c_p^{-i-1}+(z_{i+1}+y_k)\, {}^{k-1}c_{p-1}^{-i}.
\end{equation}
If $i\neq p-k$ or $i=0$ there is nothing more to prove. If 
$i=p-k>0$ the result follows from (\ref{imeq}) and the fact that
\[
e^k_k(Y)e^{p-k}_{p-k}(-Z) = 
e^{k-1}_{k-1}(Y)e^{p-k+1}_{p-k+1}(-Z) + (z_{p-k+1}+y_k)\, 
e^{k-1}_{k-1}(Y)e^{p-k}_{p-k}(-Z).
\]
\end{proof}

For any integer sequences $\al$, $\be$ and composition $\rho$, let 
\[
{}^\rho\wh{c}_\al^\be:= {}^{\rho_1}\wh{c}_{\al_1}^{\be_1}
{}^{\rho_2}\wh{c}_{\al_2}^{\be_2}\cdots
\]
where, for each $i\geq 1$, 
\[
{}^{\rho_i}\wh{c}_{\al_i}^{\be_i}:= {}^{\rho_i}c_{\al_i}^{\be_i} + 
\begin{cases}
(-1)^ie^{\rho_i}_{\rho_i}(Y)e^{\al_i-\rho_i}_{\al_i-\rho_i}(-Z) & 
\text{if $\be_i = \rho_i - \al_i < 0$}, \\
0 & \text{otherwise}.
\end{cases}
\]

\begin{defn}
\label{DSdef} 
Let $\rho$ be a composition and $\al=(\al_1,\ldots,\al_\ell)$,
$\be=(\be_1,\ldots,\be_\ell)$ be two integer vectors. Let $R$ be any
raising operator, set $\nu:=R\al$, and define
\[
R \star {}^\rho\wh{c}^\be_{\al} = {}^\rho\ov{c}^\be_{\nu} :=
{}^{\rho_1}\ov{c}_{\nu_1}^{\be_1}\cdots{}^{\rho_\ell}\ov{c}^{\be_\ell}_{\nu_\ell}
\]
where for each $i\geq 1$, 
\[
{}^{\rho_i}\ov{c}_{\nu_i}^{\be_i}:= 
\begin{cases}
{}^{\rho_i}c_{\nu_i}^{\be_i} & \text{if $i\in\supp(R)$}, \\
{}^{\rho_i}\wh{c}_{\nu_i}^{\be_i} & \text{otherwise}.
\end{cases}
\]
Set
\begin{equation}
\label{raisepfaff}
{}^\rho\wh{P}^\be_{\al}(c) := 2^{-\ell}\, \RR \star {}^\rho\wh{c}^\be_{\al}.
\end{equation}
\end{defn}

\begin{lemma}
\label{Dlem2}
Suppose that $\be_i= \rho_i-\al_i<0$ for every $i\in [1,\ell]$.
Then we have
\begin{equation}
\label{pfaffraise}
{}^\rho\wh{P}^\be_{\al}(c) = 2^{-\ell}\, 
\Pf({}^{\rho_i,\rho_j}\wh{P}^{\be_i,\be_j}_{\al_i,\al_j}(c))_{i<j}.
\end{equation}
In addition, if $\al_j=\al_{j+1}$ and $\be_j=\be_{j+1}$ for some $j\in 
[1,\ell-1]$, then 
\begin{equation}
\label{pfaffzero}
{}^\rho\wh{P}^\be_{\al}(c) = 0.
\end{equation}
\end{lemma}
\begin{proof}
Arguing as in \S \ref{spfan} for the double Schur $P$-functions, one
shows that the raising operator expression $\RR \star
{}^\rho\wh{c}^\be_{\al}$ in (\ref{raisepfaff}) may be written formally
as the Schur-type Pfaffian in (\ref{pfaffraise}). The proof of the
vanishing statement (\ref{pfaffzero}) is similar to
\cite[Prop.\ 5.4]{IM}. Suppose that $k,r\geq 0$, let $\xi$ be a
formal variable, and
\[
F(\xi):=\sum_{p=0}^\infty {}^kc_p^{-r}\xi^p = 
\prod_{i=1}^{\infty}\frac{1+x_i\xi}{1-x_i\xi}\prod_{j=1}^k(1+y_j\xi)
\prod_{m=1}^r(1-z_m\xi)
\]
be the generating function for the sequence $\{{}^kc_p^{-r}\}_{p\geq 0}$. 
Then we clearly have
\begin{equation}
\label{sqeq}
F(\xi)F(-\xi) = \prod_{j=1}^k(1-y^2_j\xi^2)
\prod_{m=1}^r(1-z^2_m\xi^2).
\end{equation}
Equating the like even powers of $\xi$ on both sides of (\ref{sqeq}) gives
\[
\frac{1-R_{12}}{1+R_{12}}\, {}^{k,k}c^{-r,-r}_{p,p} = 
\begin{cases}
e_k(y_1^2,\ldots,y_k^2)e_r(z_1^2,\ldots,z_r^2) & \text{if $p=k+r$}, \\
0 & \text{if $p>k+r$}.
\end{cases}
\]
We deduce that if $p\geq k+r$, then 
\[
\frac{1-R_{12}}{1+R_{12}}\star {}^{k,k}\wh{c}^{-r,-r}_{p,p} = 0,
\]
and therefore that ${}^{k,k}\wh{P}^{-r,-r}_{p,p}(c)=0$.  Equation
(\ref{pfaffzero}) now follows using (\ref{pfaffraise}) and the
alternating properties of Pfaffians, as in \cite[\S 1]{Ka} and
\cite[\S 4]{IM}.
\end{proof}

The next result is equivalent to Ikeda, Mihalcea, and Naruse's Pfaffian 
formula for $\DS_{\wt{w}_0}$ from \cite[Thm.\ 1.2]{IMN1}.

\begin{prop}
\label{Dgoalprop}
For any integer $n\geq 1$, we have
\begin{equation}
\label{Dgoal}
\DS_{\wt{w}_0}(X\, ;Y,Z) = 
{}^{\delta_{n-1}}\wh{P}_{2\delta_{n-1}}^{-\delta_{n-1}}(c)
\end{equation}
in $\Gamma'[Y,Z]$.
\end{prop}
\begin{proof}
It follows from Propositions \ref{Pprop} and \ref{Dtopfirst} and
Corollary \ref{Pcor}(b) that (\ref{Dgoal}) holds when $n$ is even.
Assume that $n$ is even for the rest of this proof. Since we have
\[
s_{n-1}\cdots s_1 s_\Box s_2\cdots s_{n-1} \wt{w}_0^{(n)} = \wt{w}_0^{(n-1)}
\]
in $\wt{W}_n$, we get using (\ref{Dddeq}) 
a corresponding equation of divided differences
\[
\DS_{\wt{w}_0^{(n-1)}} = 
(\partial_{n-1}\cdots \partial_2)(\partial_{\Box}\partial_1)
(\partial_2 \cdots \partial_{n-1})
\DS_{\wt{w}_0^{(n)}}
\]
in $\Gamma'[Y,Z]$. We therefore obtain the equality
\begin{equation}
\label{nevdd}
\DS_{\wt{w}_0^{(n-1)}} = 2^{1-n}
(\partial_{n-1}\cdots \partial_2)(\partial_{\Box}\partial_1)
(\partial_2 \cdots \partial_{n-1})
\left(\RR \star {}^{\delta_{n-1}}\wh{c}_{2\delta_{n-1}}^{-\delta_{n-1}}\right).
\end{equation}

We next compute the action of the divided differences on the 
right hand side of (\ref{nevdd}).
By using Lemmas \ref{Dlem1} and \ref{Dlem2} and arguing as in \S
\ref{ttotheta}, we see that
\[
(\partial_2 \cdots \partial_{n-1})
\left(\RR \star {}^{\delta_{n-1}}\wh{c}_{2\delta_{n-1}}^{-\delta_{n-1}}\right) 
= \RR \star
{}^{(n-2,n-3,\ldots,1,1)}\wh{c}_{(2n-3,2n-5,\ldots,3,2)}^{-\delta_{n-1}}.
\]
For any integer vector $\al := (\al_1\ldots, \al_{n-1})$, we have
\[
\partial_1({}^{(n-2,n-3,\ldots,1,1)}\wh{c}_{\al}^{-\delta_{n-1}}) = 
{}^{\delta_{n-2}}\wh{c}_{(\al_1,\ldots,\al_{n-2})}^{(1-n,2-n,\ldots,-2)}
{}^1g^0_{\al_{n-1}-1},
\]
where 
\[
{}^1g^0_p= \begin{cases}
2\, {}^1f_1 & \text{if $p=1$}, \\
{}^1\wh{c}^0_p & \text{otherwise}.
\end{cases}
\]
According to \cite[\S 1]{T7}, for $k\geq 1$, we have 
\begin{equation}
\label{peq}
\partial_\Box({}^kc_p) = 2\,{}^kc^2_{p-1} \ \ \ \mathrm{and} \ \ \ 
\partial_\Box({}^kb_k) = \partial_\Box({}^k\wt{b}_k) = {}^kc^2_{k-1}.
\end{equation}
We therefore also have $\partial_\Box({}^1f_1) = 1$ by (\ref{peq}) and
$\partial_\Box({}^1\wh{c}^0_p)=0$ for $p\leq 0$. It follows that
\begin{equation}
\label{treq}
\partial_\Box\partial_1({}^{(n-2,n-3,\ldots,1,1)}\wh{c}_{\al}^{-\delta_{n-1}}) =
\begin{cases} 
2\cdot{}^{\delta_{n-2}}\wh{c}_{(\al_1,\ldots,\al_{n-2})}^{(1-n,2-n,\ldots,-2)} & 
\text{if $\al_{n-1}=2$}, \\
0 & \text{if $\al_{n-1} < 2$}.
\end{cases}
\end{equation}
We deduce from (\ref{treq}) that $\partial_\Box\partial_1$ commutes with 
the action of the raising operators $R$ in the expansion of $\RR$ in its
$\star$-action on ${}^{(n-2,n-3,\ldots,1,1)}\wh{c}_{(2n-3,2n-5,\ldots,3,2)}^{-\delta_{n-1}}$,
and hence that 
\[
(\partial_\Box\partial_1 \cdots \partial_{n-1})
\left(\RR \star {}^{\delta_{n-1}}\wh{c}_{2\delta_{n-1}}^{-\delta_{n-1}}\right) 
= 2\, \RR \star {}^{\delta_{n-2}}\wh{c}_{(2n-3,2n-5,\ldots,3)}^{(1-n,2-n,\ldots,-2)}.
\]
We continue applying Lemma \ref{Dddlem}(b) to compute the action of 
$\partial_{n-1}\cdots \partial_2$ on $\RR \star
{}^{\delta_{n-2}}\wh{c}_{(2n-3,2n-5,\ldots,3)}^{(1-n,2-n,\ldots,-2)}$, 
to conclude that
\[
\DS_{\wt{w}_0^{(n-1)}} = 2^{2-n}\, (\partial_{n-1}\cdots \partial_2)
\left(\RR \star
{}^{\delta_{n-2}}\wh{c}_{(2n-3,2n-5,\ldots,3)}^{(1-n,2-n,\ldots,-2)}\right) =
2^{2-n}\, \RR \star
{}^{\delta_{n-2}}\wh{c}_{2\delta_{n-2}}^{-\delta_{n-2}},
\]
and hence that (\ref{Dgoal}) holds for all $n\geq 1$, as required.
\end{proof}

\begin{cor}
\label{Dtopcor}
If $n$ is odd, then we have 
\begin{equation}
\label{DSeq2}
\DS_{\wt{w}_0}(X\, ;Y,Z) = 
\sum_{\la\subset\delta^*_n} P_{\delta_{n-2}+\la}(X) 
S^{(\delta_{n-1},\delta_{n-1})}_{\delta^*_n/\la}(e(Y,-Z))
\end{equation}
in $\Gamma'[Y,Z]$.
\end{cor}
\begin{proof}
This follows by combining Proposition \ref{Dgoalprop} with
(\ref{Pkk}) and Corollary \ref{Pcor}(a).
\end{proof}

Note the similarity between formulas (\ref{s0w0eq}) and
(\ref{DSeq2}). It would be interesting to expose a more direct
argument connecting the two to each other.

\subsection{The Schubert polynomials indexed by maximal elements}
\label{ttoeta}

Consider a sequence $\fraka \ :\ a_1<\cdots < a_p$ of elements of
$\N_{\Box}$ with $a_p<n$. The sequence $\fraka$ parametrizes a
parabolic subgroup $\wt{W}_\fraka$ of $\wt{W}_n$, which is generated
by the simple reflections $s_i$ for $i\notin\{a_1,\ldots, a_p\}$. In
type D, we will only consider sequences $\fraka$ with $a_1 \neq 1$,
since these suffice to parametrize all the relevant homogeneous spaces
and degeneracy loci, up to isomorphism.\footnote{This convention is
simpler than the one used in \cite[\S 6]{T5} and \cite[\S 5.3]{T6}.}
This claim is due to
the natural involution of the Dynkin diagram of type $\mathrm{D}_n$.
Geometrically, it is explained by the fact that any isotropic subspace
$E_{n-1}$ of $\C^{2n}$ (equipped with an orthogonal form) with
$\dim(E_{n-1})= n-1$ can be uniquely extended to a two-step flag
$E_{n-1}\subset E_n$ with $E_n$ maximal isotropic and in a given
family (compare with \cite[\S 6.3.2]{T6}).

Define the set $\wt{W}_n^\fraka$ by
\[
\wt{W}_n^\fraka := 
\{w\in \wt{W}_n\ |\ \ell(ws_i) = \ell(w)+1,\  \forall\, i \notin
\{a_1,\ldots, a_p\}\}
\]
and let $\wt{w}_0(\fraka)$ denote the longest element in $\wt{W}_n^\fraka$. 
We have
\[
\wt{w}_0(\fraka)=\begin{cases}
\ov{a_2}\cdots\ov{2}\wh{1}\ov{a_3}\cdots \ov{a_2+1}\cdots\ov{n}\cdots\ov{a_p+1}
& \text{if $a_1=\Box$}, \\
\wh{1}2\cdots a_1\ov{a_2}\cdots\ov{a_1+1}\cdots\ov{n}\cdots\ov{a_p+1}
& \text{if $a_1\neq \Box$},
\end{cases}
\]
where $\wt{1}$ is equal to either $1$ or $\ov{1}$, specified so that 
$\wt{w}_0(\fraka)$ contains an even number of barred integers.

Fix an element $k\in \N_\Box$ with $\Box\leq k < n$, and set
$\wt{W}_n^{(1)}:=\wt{W}_n^{(\Box,1)}$. The elements of the set
$\wt{W}_n^{(k)}$ are the {\em $k$-Grassmannian} elements of
$\wt{W}_n$.  Let $\wt{w}^{(k,n)}=\wh{1}2\cdots k\, \ov{n}\cdots
\ov{k+1}$ denote the longest element of $\wt{W}_n^{(k)}$.  Following \cite[\S
  3.2]{T7}, we will require a formula analogous to (\ref{Dgoal}) for
the Schubert polynomial $\DS_{\wt{w}^{(k,n)}}(X\,;Y,Z)$, which maps to
Kazarian's multi-Schur Pfaffian formula from
\cite[Thm.\ 1.1]{Ka}. Corresponding Pfaffian formulas for the Schubert
polynomials $\DS_{w_0(\fraka)}(X\,; Y,Z)$ were obtained in \cite{AF1}.

\begin{prop}
\label{wtknprop}
We have
\[
\DS_{\wt{w}^{(k,n)}}(X\,;Y,Z)=
{}^{(k,\ldots, k)}\wh{P}_{(n+k-1,\ldots, 2k)}^{(1-n,\ldots,-k)}(c).
\]
in $\Gamma'[Y,Z]$.
\end{prop}
\begin{proof}
Let $\wt{v}^{(k,n)}=\ov{k}\cdots \ov{2}\wh{1} \,\ov{n}\cdots \ov{k+1}$ be the
longest element in $\wt{W}_n^{(\Box,k)}$. Then we have a reduced factorization
$\wt{w}_0= v_1v_2 \wt{v}^{(k,n)}$, where $v_1$ and $v_2$ are defined by 
(\ref{eqv1}) and (\ref{eqv2}), as in the type C case.
Using the equations (\ref{Dddeq}), we obtain the relation
\[
\DS_{\wt{v}^{(k,n)}} = \partial_{n-1} (\partial_{n-2}\partial_{n-1}) \cdots
(\partial_{k+1} \cdots \partial_{n-1})
\cdot\partial_{k-1}(\partial_{k-2}\partial_{k-1}) \cdots (\partial_1
\cdots \partial_{k-1}) \DS_{\wt{w}_0}.
\]

Assume that $k \geq 2$, as the proof is easier if $k\in \{\Box,1\}$. 
First, using Lemmas \ref{Dddlem} and \ref{Dlem1}, for any $p\in \Z$
we have 
\begin{equation}
\label{k2eq}
\partial_{k-1} {}^{k-1}\wh{c}^{1-k}_p = {}^{k-1}\wh{c}^{2-k}_{p-1}  = 
{}^{k-2}\wh{c}^{1-k}_{p-1}+(z_{k-1}+y_{k-1}){}^{k-2}\wh{c}^{2-k}_{p-2}.
\end{equation}
The Leibnitz rule and (\ref{k2eq}) imply that for any
integer vector $\al=(\al_1,\ldots,\al_{n-1})$, we have
\[
\partial_{k-1} {}^{\delta_{n-1}}\wh{c}^{-\delta_{n-1}}_{\alpha} = 
{}^{\delta_{n-1}-\epsilon_{n+1-k}}\wh{c}^{-\delta_{n-1}}_{\al-\epsilon_{n+1-k}}+
(z_{k-1}+y_{k-1}){}^{\delta_{n-1}-\epsilon_{n+1-k}}
\wh{c}^{-\delta_{n-1}+\epsilon_{n+1-k}}_{\al - 2\epsilon_{n+1-k}}.
\]
The last equation implies that $\partial_{k-1}$ commutes with $\star$ action of 
the raising operators $R$ in the expansion of $\RR$ on 
${}^{\delta_{n-1}}\wh{c}^{-\delta_{n-1}}_{2\delta_{n-1}}$. We deduce from this
and Lemma \ref{Dlem2} that 
\begin{align*}
\partial_{k-1}
{}^{\delta_{n-1}}\wh{P}_{2\delta_{n-1}}^{-\delta_{n-1}}(c) &= 
{}^{\delta_{n-1}-\epsilon_{n+1-k}}\wh{P}^{-\delta_{n-1}}_{2\delta_{n-1}-\epsilon_{n+1-k}}(c) \\ 
&= {}^{(n-1,\ldots, k, k-2,k-2,\ldots, 1)}
\wh{P}_{(2n-2,\ldots,2k, 2k-3, 2k-4,\ldots,2)}^{-\delta_{n-1}}(c).
\end{align*}
Iterating this calculation gives
\[
(\partial_1\cdots \partial_{k-1})\DS_{\wt{w}_0} = 
{}^{(n-1,\ldots, k, k-2,k-3,\ldots, 0)}\wh{P}_{(2n-2,\ldots, 2k, 2k-3, 2k-5,\ldots,3,1)}^{-\delta_{n-1}}(c)
\]
and furthermore
\[
\partial_{k-1}(\partial_{k-2}\partial_{k-1}) \cdots (\partial_1 \cdots
\partial_{k-1}) \DS_{\wt{w}_0} = {}^{(n-1,\ldots, k, 0,0,\ldots,
  0)}\wh{P}_{(2n-2,\ldots, 2k, k-1, k-2,\ldots,1)}^{-\delta_{n-1}}(c).
\]
Applying the operator $ \partial_{n-1} (\partial_{n-2}\partial_{n-1}) \cdots
(\partial_{k+1} \cdots \partial_{n-1})$ to the latter, 
we similarly get
\[
\DS_{\wt{v}^{(k,n)}} =
{}^{(k,k,\ldots, k, 0,0,\ldots, 0)}\wh{P}_{(n+k-1,n+k-2,\ldots, 2k, k-1,\ldots,1)}^{-\delta_{n-1}}(c).
\]

We also have $\wt{v}^{(k,n)} = (s_\Box s_2\cdots s_{k-1}) \cdots 
(s_\Box s_2)s_\Box \wt{w}^{(k,n)}$, and hence (\ref{Dddeq}) gives
\[
\DS_{\wt{w}^{(k,n)}} = \partial_\Box (\partial_2\partial_\Box) \cdots (\partial_{k-1}
\cdots\partial_2\partial_\Box)\DS_{\wt{v}^{(k,n)}}.
\]
Finally, using the fact that 
\[
{}^{(\rho,r)}\wh{P}^{(\be,b)}_{(\al,0)}(c) = {}^{\rho}\wh{P}^{\be}_{\al}(c),
\]
we compute that
\begin{align*}
\DS_{\wt{w}^{(k,n)}} &=  \partial_\Box (\partial_2\partial_\Box) \cdots (\partial_{k-1}
\cdots\partial_2\partial_\Box)\DS_{\wt{v}^{(k,n)}} \\
&= \partial_\Box (\partial_2\partial_\Box) \cdots (\partial_{k-2}
\cdots\partial_2\partial_\Box)
{}^{(k,\ldots, k, 0,\ldots, 0)}\wh{P}_{(n+k-1,\ldots, 2k, k-2, \ldots,1)}
^{(1-n,\ldots,-k,2-k\ldots,0)}(c) \\
&={}^{(k,\ldots, k)}\wh{P}_{(n+k-1,\ldots, 2k)}^{(1-n,\ldots,-k)}(c).\\
\end{align*}
\end{proof}

More generally, using similar arguments to those above, we can prove that
\[
\DS_{\wt{w}_0(\fraka)}(X\,; Y,Z) = 
{}^{\rho(\fraka)}\wh{P}^{\be(\fraka)}_{\la(\fraka)}(c),
\]
where $\la(\fraka)$, $\be(\fraka)$, and $\rho(\fraka)$ denote the sequences
\[
\la(\fraka)=(n+a_p-1,\ldots,2a_p,\ldots,
a_i+a_{i+1}-1,\ldots,2a_i,\ldots,a_1+a_2-1,\ldots,2a_1)\,;
\]
\[
\be(\fraka)=(1-n,\ldots,-a_p,\ldots,
1-a_{i+1},\ldots,-a_i,\ldots,1-a_2,\ldots,-a_1)\,;
\]
and
\[
\rho(\fraka)=(a_p^{n-a_p},\ldots,a_i^{a_{i+1}-a_i},\ldots,a_1^{a_2-a_1}).
\]

\subsection{Eta polynomials}
\label{eps}

According to \cite{BKT1,BKT3}, a {\em typed $k$-strict partition} is a
pair consisting of a $k$-strict partition $\la$ together with an
integer $\type(\la)\in \{0,1,2\}$, which is positive if and only if
$\la_j=k$ for some index $j$. We assume that $k>0$ here, although it
is straightforward to include the case $k=0$, where a typed $0$-strict
partition is simply a strict partition (of type zero). There is a
bijection between the $k$-Grassmannian elements $w$ of $\wt{W}_\infty$
and typed $k$-strict partitions $\la$, under which the elements in
$\wt{W}_n$ correspond to typed partitions whose diagram fits inside an
$(n-k)\times (n+k-1)$ rectangle, obtained as follows. If the element
$w$ corresponds to the typed partition $\la$, then for each $j\geq 1$,
\[
\la_j=\begin{cases} 
|w_{k+j}|+k-1 & \text{if $w_{k+j}<0$}, \\
\#\{p\leq k\, :\, |w_p|> w_{k+j}\} & \text{if $w_{k+j}>0$}
\end{cases}
\]
while $\type(\la)>0$ if and only if $|w_1|>1$, and in this case
$\type(\la)$ is equal to $1$ or $2$ depending on whether $w_1>0$ or
$w_1<0$, respectively. To any typed $k$-strict partition $\la$, we
associate a finite set of pairs $\cC(\la)$ and a sequence
$\beta(\la)=\{\beta_j(\la)\}_{j\geq 1}$ using this bijection and the
same equations (\ref{Cweq}) and (\ref{css}) as in the type C case.

The Schubert polynomials indexed by $k$-Grassmannian elements are
represented by {\em eta polynomials}. For any typed $k$-strict
partition $\la$, the raising operator expression $R^\la$ is defined by
equation (\ref{Req}), as before.  Let $\ell$ denote the length of
$\la$, let $\ell_k(\la)$ denote the number of parts $\la_i$ which are
strictly greater than $k$, let $m:=\ell_k(\la)+1$ and $\be:=\be(\la)$.
If $R:=\prod_{i<j} R_{ij}^{n_{ij}}$ is any raising operator, denote by
$\supp_m(R)$ the set of all indices $i$ and $j$ such that $n_{ij}>0$
and $j<m$, and set $\nu:=R\la$. If $\type(\la)=0$, then define
\[
R \star \wh{c}^\be_{\la} = \ov{c}^\be_{\nu} :=
\ov{c}_{\nu_1}^{\be_1}\cdots\ov{c}^{\be_\ell}_{\nu_\ell}
\]
where for each $i\geq 1$, 
\[
\ov{c}_{\nu_i}^{\be_i}:= 
\begin{cases}
{}^kc_{\nu_i}^{\be_i} & \text{if $i\in\supp_m(R)$}, \\
{}^k\wh{c}_{\nu_i}^{\be_i} & \text{otherwise}.
\end{cases}
\]
For any $p,r\in \Z$ and $s\in \{0,1\}$, 
define 
\[
a^s_p:= \frac{1}{2}{}^kc_p+\omega_p^s, \ \ \
b^s_k:= {}^kb_k+\omega_k^s, \ \ \, \text{and} \ \ \,
\wt{b}^s_k:= {}^k\wt{b}_k+\omega_k^s,
\]
where $\dis\omega_p^s=\omega_p^s(X\, ;\, Y,Z):=\sum_{j=1}^p
{}^kc_{p-j} \, h^s_j(-Z)$.

If $\type(\la)>0$ and $R$ involves any factors $R_{ij}$ with $i=m$ or
$j=m$, then define
\[
R \star \wh{c}^\be_{\la} := \ov{c}_{\nu_1}^{\be_1} \cdots 
\ov{c}_{\nu_{m-1}}^{\be_{m-1}} \, a^{\be_m}_{\nu_m} \, c_{\nu_{m+1}}^{\be_{m+1}}
\cdots c^{\be_\ell}_{\nu_\ell},
\]
where $c_{\nu_i}^{\be_i}:={}^kc_{\nu_i}^{\be_i}$ for each $i$.
If $R$ has no such factors, then define
\[
R \star \wh{c}^\be_{\la} := \begin{cases}
\ov{c}_{\nu_1}^{\be_1} \cdots 
\ov{c}_{\nu_{m-1}}^{\be_{m-1}} \, b^{\be_m}_k \, c_{\nu_{m+1}}^{\be_{m+1}}
\cdots c^{\be_\ell}_{\nu_\ell} & 
\text{if  $\,\type(\la) = 1$}, \\
\ov{c}_{\nu_1}^{\be_1} \cdots
\ov{c}_{\nu_{m-1}}^{\be_{m-1}} \, \wt{b}^{\be_m}_k \, c_{\nu_{m+1}}^{\be_{m+1}} 
\cdots c^{\be_\ell}_{\nu_\ell} 
& \text{if  $\,\type(\la) = 2$}.
\end{cases}
\]
Following \cite{T7}, define the {\em double eta polynomial}
$\Eta_\la(X\,;Y_{(k)},Z)$ by
\begin{equation}
\label{etadef}
\Eta_\la(X\,;Y_{(k)},Z) := 2^{-\ell_k(\la)}R^\la\star\wh{c}^{\be(\la)}_{\la}.
\end{equation}
The single eta polynomial $\Eta_\la(X\,;Y_{(k)})$ of \cite{BKT3} is given by
\[
\Eta_\la(X\,;Y_{(k)}):=\Eta_\la(X\,;Y_{(k)},0).
\] 
As in \S \ref{tps}, we note that we are working here with the images
in the ring $\Gamma'[Y,Z]$ of the eta polynomials $\Eta_\la(c)$ and
$\Eta_\la(c\, |\, t)$ from \cite{T6, T7}.

Fix a rank $n$ and let $$\wt{\la}_0:=(n+k-1,n+k-2,\ldots,2k)$$ be the
typed $k$-strict partition associated to the $k$-Grassmannian element 
$\wt{w}^{(k,n)}$ of maximal length in $\wt{W}_n$. We deduce from 
Proposition \ref{wtknprop} and (\ref{etadef}) that
\begin{equation}
\label{deqt}
\DS_{\wt{w}^{(k,n)}}(X\,; Y,Z) = \Eta_{\wt{\la}_0}(X\,; Y_{(k)},Z).
\end{equation}

Using raising operators, it is shown in \cite[Prop.\ 5]{T7} that if
$\la$ and $\mu$ are typed $k$-strict partitions such that 
$|\la|=|\mu|+1$ and $w_\la=s_iw_{\mu}$ for some simple
reflection $s_i\in \wt{W}_\infty$, then we have
\begin{equation}
\label{pareqD}
\partial_i\Eta_\la(X\,; Y_{(k)},Z)  = \Eta_{\mu}(X\,; Y_{(k)},Z)
\end{equation}
in $\Gamma'[Y,Z]$.\footnote{The paper \cite{T7} assumes that $k>0$,
but the proofs also work (and are simpler) when $k=0$.} It follows
easily from (\ref{deqt}) and (\ref{pareqD}) that for any typed
$k$-strict partition $\la$ with associated $k$-Grassmannian element
$w_\la$, we have
\[
\DS_{w_\la}(X\,; Y,Z) = \Eta_{\la}(X\,; Y_{(k)},Z)
\]
in $\Gamma'[Y,Z]$. In particular, we recover the equality
\begin{equation}
\label{deqtsing}
\DS_{w_\la}(X\,; Y) =  \Eta_{\la}(X\,; Y_{(k)})
\end{equation}
in $\Gamma'[Y]$ from \cite[Prop.\ 6.3]{BKT3} for the single Schubert and eta
polynomials.

\subsection{Mixed Stanley functions and splitting formulas}
\label{msfsfsD}

For any $w\in \wt{W}_\infty$, the {\em double mixed Stanley function}
$I_w(X\,;Y/Z)$ is defined by the equation
\[
I_w(X\,;Y/Z) := \langle \tilde{A}(Z)D(X)A(Y), w\rangle = 
\sum_{uv\om=w}G_{u^{-1}}(-Z)E_v(X)G_\om(Y),
\]
where the sum is over all reduced factorizations $uv\om=w$ with
$u,\om\in S_\infty$. The single mixed Stanley function $I_w(X\,;Y)$ from
\cite[\S 6]{T5} is
given by setting $Z=0$ in $I_w(X\,;Y/Z)$. 

Fix an element $k\in \N_\Box$. If $k\geq 2$, we say that an element
$w\in \wt{W}_\infty$ is {\em increasing up to $k$} if $|w_1|<
w_2<\cdots < w_k$. Furthermore, we adopt the convention that every
element of $\wt{W}_\infty$ is increasing up to $\Box$ and increasing
up to $1$. If $w$ is increasing up to $k$, then there is an analogue
of (\ref{keyidC}) for the {\em restricted mixed Stanley function}
$I_w(X\,;Y_{(k)})$, which is obtained from $I_w(X;Y)$ after setting
$y_i=0$ for $i>k$. In this case, according to \cite[Eqn.\ (33)]{T5},
we have
\begin{equation}
\label{keyidD}
\DS_w(X\,;Y)= 
\sum_{v(1_k\times \om) = w}I_v(X\,;Y_{(k)})\AS_\om(y_{k+1},y_{k+2},\ldots),
\end{equation}
where the sum is over all reduced factorizations $v(1_k\times \om) =
w$ in $\wt{W}_\infty$ with $\om\in S_\infty$. Akin to \S
\ref{ssfsfsA} and \S \ref{msfsfsC}, equation (\ref{keyidD}) has a
double version: let $I_v(X\,;Y_{(k)}/Z_{(\ell)})$ be the power series
obtained from $I_v(X\,;Y/Z)$ by setting $y_i=z_j=0$ for all $i>k$ and
$j>\ell$. Then if $w$ is increasing up to $k$ and $w^{-1}$ is
increasing up to $\ell$, we have
\begin{equation}
\label{keyidD2}
\DS_w(X\,;Y,Z)= 
\sum \AS_{u^{-1}}(-Z_{>\ell}) I_v(X\,;Y_{(k)}/Z_{(\ell)})\AS_\om(Y_{>k}),
\end{equation}
where the sum is over all reduced factorizations $(1_{\ell}\times
u)v(1_k\times \om) = w$ in $\wt{W}_\infty$ with $u,\om\in S_\infty$.

We say that an element $w\in \wt{W}_\infty$ is {\em compatible} with
the sequence $\fraka \, :\, a_1 < \cdots < a_p$ of elements of
$\N_\Box$ if all descent positions of $w$ are contained in $\fraka$
(following \S \ref{ttoeta}, we assume that $a_1\neq 1$). Let $\frakb
\, :\, b_1 < \cdots <b_q$ be a second sequence of elements of
$\N_\Box$, and suppose that $w$ is compatible with $\fraka$ and
$w^{-1}$ is compatible with $\frakb$. The notion of a reduced
factorization $u_1\cdots u_{p+q-1} = w$ compatible with $\fraka$,
$\frakb$ and the sets of variables $Y_i$ and $Z_j$ for $i,j\geq 1$ are
defined exactly as in \S \ref{msfsfsC}.

\begin{prop}
\label{split2Dprop}
Suppose that $w$ and $w^{-1}$ are compatible with $\fraka$ and 
$\frakb$, respectively. Then the
Schubert polynomial $\DS_w(X\,;Y,Z)$ satisfies
\[
\DS_w = \sum
G_{u_1}(0/Z_q)\cdots G_{u_{q-1}}(0/Z_2)
I_{u_q}(X\,;Y_1/Z_1)G_{u_{q+1}}(Y_2) \cdots G_{u_{p+q-1}}(Y_p)
\]
summed over all reduced factorizations $u_1\cdots u_{p+q-1} = w$
compatible with $\fraka$, $\frakb$.  
\end{prop}
\begin{proof}
The result is shown by using (\ref{keyidD2}) and iterating the
identity (\ref{keyid}).
\end{proof}

If $w$ is increasing up to $k$, then the following generalization of
equation (\ref{DStan0}) holds (see \cite[Eqn.\ (32)]{T5}):
\begin{equation}
\label{DmStan}
I_w(X\,;Y_{(k)}) = \sum_{\la\, :\, |\la| = \ell(w)} d^w_{\la}\,\Eta_{\la}(X;Y_{(k)}),
\end{equation}
where the sum is over typed $k$-strict partitions $\la$ with $|\la| =
\ell(w)$. The {\em mixed Stanley coefficients} $d^w_{\la}$ in
(\ref{DmStan}) are nonnegative integers which have a combinatorial
interpretation, as in \S \ref{msfsfsC}.  The proof of (\ref{DmStan})
in \cite{T5} uses (\ref{deqtsing}), and is similar to the type C case.

Assume that $w$ is increasing up to $k$ and 
$w^{-1}$ is increasing up to $\ell$. Then the 
(restricted) double mixed Stanley function
$I_w(X\,;Y_{(k)}/Z_{(\ell)})$ satisfies
\begin{align}
\label{Ieq1}
I_w(X\,;Y_{(k)}/Z_{(\ell)}) 
&= \sum_{uv=w}G_{u^{-1}}(-Z_{(\ell)})I_v(X\,;Y_{(k)}) \\
\label{Ieq2}
&= \sum_{uv=w^{-1}}G_{u^{-1}}(Y_{(k)})I_v(X\,;-Z_{(\ell)}),
\end{align}
where the sums are over reduced factorizations as shown, with $u\in
S_\infty$. We can now use equations (\ref{Geq}) and (\ref{DmStan}) in
(\ref{Ieq1}) and (\ref{Ieq2}) to obtain two expansions of
$I_w(X\,;Y_{(k)}/Z_{(\ell)})$ as a positive sum of products of Schur
$S$-polynomials with eta polynomials, as in Example \ref{Jex}.

\begin{thm}[\cite{T5}, Cor.\ 3]
Suppose that $w$ is compatible with $\fraka$ and $w^{-1}$ is compatible
with $\frakb$, where $b_1=\Box$. Then we have
\[
\DS_w = \sum_{\underline{\la}} 
g^w_{\underline{\la}}\,
s_{\la^1}(0/Z_q)\cdots s_{\la^{q-1}}(0/Z_2)
\Eta_{\la^q}(X\,;Y_1)s_{\la^{q+1}}(Y_2)\cdots s_{\la^{p+q-1}}(Y_p)
\]
summed over all sequences of partitions
$\underline{\la}=(\la^1,\ldots,\la^{p+q-1})$ with $\la^q$ $a_1$-strict and typed,
where 
\[
g^w_{\underline{\la}} := \sum_{u_1\cdots u_{p+q-1} = w}
a_{\la^1}^{u_1}\cdots a_{\la^{q-1}}^{u_{q-1}}
d_{\la^q}^{u_q}a_{\la^{q+1}}^{u_{q+1}}\cdots a_{\la^{p+q-1}}^{u_{p+q-1}}
\]
summed over all reduced factorizations $u_1\cdots u_{p+q-1} = w$
compatible with $\fraka$, $\frakb$.
\end{thm}
\begin{proof}
This follows from Proposition \ref{split2Dprop} by using 
equations (\ref{Geq}) and (\ref{DmStan}).
\end{proof}

\section{Geometrization}
\label{geometriz}

In this section, we discuss the precise way in which the Schubert
polynomials given in the previous sections represent degeneracy loci
of vector bundles in the sense of \cite{Fu2}.  This has been addressed
in earlier work (see \cite[\S 4]{T5} and \cite[\S 6, 7]{T6}), but our
aim here is to provide some more detailed historical comments, which
include the author's papers \cite{T2, T3}. We restrict attention to
the symplectic case, as the orthogonal types are analogous to type C,
and the situation in type A has been understood since \cite{Fu1}.
Throughout the section $\XX:=(\x_1,\x_2,\ldots)$ and $\YY:=(\y_1,\y_2,
\ldots)$ will denote two sequences of commuting independent variables,
and for every integer $n\geq 1$, we set $\XX_n:=(\x_1,\ldots,\x_n)$ and
$\YY_n:=(\y_1,\ldots, \y_n)$.

Consider a vector bundle $E\to M$ of rank $2n$ over a smooth algebraic
variety $M$, equipped with an everywhere nondegenerate skew-symmetric 
form $E\otimes E\to \C$. Assume that we are given two complete flags of 
subbundles of $E$
\begin{equation}
\label{seqs}
E_\bull\ :\ 0 \subset E_1\subset \cdots \subset E_{2n}=E \ \ \mathrm{and} \ \ 
F_\bull\ :\ 0 \subset F_1\subset \cdots \subset F_{2n}=E
\end{equation}
with $\rank E_i = \rank F_i = i$ for each $i$, while
$E_{n+i}=E_{n-i}^{\perp}$ and $F_{n+i}=F_{n-i}^{\perp}$ for $0\leq i < n$.
For any $w$ in the Weyl group $W_n$, we have the degeneracy locus
\begin{equation}
\label{degdef}
\X_w:=\{x\in M\ |\ \dim(E_i(x)\cap F_j(x)) \geq d_w(i,j) \ \, \forall \, 
i\in [1,n],j\in [1,2n]\},
\end{equation}
where $d_w(i,j)\in \Z$ and the inequalities in (\ref{degdef}) are exactly 
those which define the Schubert variety $X_w(F_\bull)$ in the flag variety 
$\IF_n:=\Sp_{2n}/B$ (the precise values $d_w(i,j)$ are given in 
\cite[\S 6.2]{T6}). We assume that $\X_w$ has pure codimension 
$\ell(w)$ in $M$, and seek a formula which expresses the cohomology
class $[\X_w]\in \HH^*(M)$ as a universal polynomial in the Chern classes
of the vector bundles $E_i$ and $F_j$.

When the vector bundles $F_j$ are trivial, the answer to the above
{\em degeneracy locus problem} coincides with the answer to the {\em
  Giambelli problem} for $\HH^*(\IF_n)$, which amounts to a theory of
(single) {\em symplectic Schubert polynomials}. In this setting, the
$E_i$ are the universal (or tautological) vector bundles over $M=\IF_n$.
The cohomology ring of $\IF_n$ has a standard Borel presentation
\cite{Bo} as a quotient ring
\begin{equation}
\label{borpre}
\HH^*(\IF_n,\Z) \cong \Z[\x_1,\ldots,\x_n]/\J_n,
\end{equation}
where the variables $\x_i$ represent the characters of the Borel
subgroup $B$, or the Chern roots of the dual of the Lagrangian
subbundle $E_n$, and $\J_n$ denotes the ideal generated by the
$W_n$-invariant polynomials of positive degree. The aim of a theory of
Schubert polynomials is to provide a combinatorially explicit and
natural set of polynomial representatives $\{\CS_w(\XX_n)\}_{w\in
  W_n}$ for the Schubert classes $\{[X_w]\}_{w\in W_n}$ in the
presentation (\ref{borpre}) of $\HH^*(\IF_n,\Z)$.

Among the many desirable attributes of the Schubert
polynomials $\CS_w(\XX_n)$, the {\em stability property} is perhaps
the most important. This states that we have
\[
\CS_{j_n(w)}(\x_1,\ldots,\x_n,0) = \CS_w(\x_1,\ldots,\x_n), \ \ \forall \,
w\in W_n,
\]
where $j_n:W_n\hookrightarrow W_{n+1}$ is the natural inclusion map of
Weyl groups. The significance of this property was already recognized
in the work of Lascoux and Sch\"utzenberger \cite{LS} on type A
Schubert polynomials, where -- together with the fact that they
represent Schubert classes -- it completely characterizes them.

The inclusions $j_n$ induce surjections
\begin{equation}
\label{embeddings}
\cdots \rightarrow \HH^*(\IF_{n+1},\Z) \rightarrow
\HH^*(\IF_n,\Z)\rightarrow \cdots
\end{equation}
and the inverse limit of the system (\ref{embeddings}) in the category
of graded rings is the {\em stable cohomology ring} $\IH(\IF)$. The
stability property implies that the symplectic Schubert polynomials
lift to give representatives $\CS_w(\XX)$ of the {\em stable Schubert
  classes} $\dis \sigma_w:=\lim_{\longleftarrow}[X_w]$, one for every
$w\in W_\infty$. Unlike the situation in type A, the $\CS_w(\XX)$ will
no longer be polynomials in $\XX$, but formal power series (see
Example \ref{linearSP} below). Moreover, a special role is played by
the subring of $\IH(\IF)$ invariant under the action of the symmetric
group $S_\infty$, whose elements are represented by symmetric power
series, and which is isomorphic to the stable cohomology ring
$\IH(\LG)$ of the {\em Lagrangian Grassmannian} $\LG(n,2n)$.

The Giambelli problem for the cohomology ring of $\LG(n,2n)$ was
solved by Pragacz \cite{P} by using the theory of Schur $Q$-functions,
and the resulting isomorphism between $\HH^*(\LG(n,2n),\Z)$ and a
certain quotient of the ring $\Gamma$ of Schur $Q$-functions was
further studied by J\'ozefiak \cite{J}.  Let $\Lambda$ denote the ring
of symmetric functions in the variables $\XX$, so that
$\Lambda=\Z[e_1(\XX),e_2(\XX),\ldots]$, and let $\I$ be the ideal of
$\Lambda$ generated by the homogeneous symmetric functions in
$\XX^2:=(\x_1^2,\x_2^2,\ldots)$ of positive degree. According to
\cite[Cor.\ 2.3]{J}, the surjective map $\eta:\Lambda\to\Gamma$ with
$\eta(e_i(\XX)):=q_i(X)$ for all $i\geq 1$ induces an isomorphism
$\Lambda/\I\cong \Gamma$.

Define a map $\phi_n:\Z[\XX] \to\Z[\XX_n]$ by $\x_i\mapsto \x_i$ for
$i\leq n$, while $\x_i\mapsto 0$ for $i>n$. If
$\Lambda_n:=\Z[e_1(\XX_n),\ldots,e_n(\XX_n)] =\Z[\XX_n]^{S_n}$ is the
ring of symmetric polynomials in $\XX_n$, then $\phi_n$ induces an
homonymous map $\Lambda \to\Lambda_n$. Setting
$\I_n:=\phi_n(\I)=\Lambda_n(e_1(\XX_n^2),\ldots,e_n(\XX_n^2))$, we
then have a commutative diagram of rings
\[
\xymatrix@M=6pt{
\Gamma \ar[r]^{\pi \ \,} & \Lambda/\I 
\ar[d]^{\phi_n} \ar[r] & \IH(\LG) \ar[d] \\
& \Lambda_n/\I_n \ar[r]^{\psi_0 \qquad} & \HH^*(\LG(n,2n)) }
\]
where the horizontal arrows are isomorphisms. The map $\psi_0$ sends
$e_i(\XX_n)$ to the $i$-th Chern class $c_i(E/E_n)$ of the universal 
quotient bundle over $\LG(n,2n)$. The resulting surjection
$\Gamma\to \HH^*(\LG(n,2n))$ maps $Q_\la(X)$ to the Schubert class $[X_\la]$,
for any strict partition $\la$ with $\la_1\leq n$, and to zero, otherwise.

Since the combinatorial theory of Schur $Q$-functions was well
understood and analogous to the type A theory of Schur $S$-functions,
the above picture provided a satisfactory way to do classical Schubert
calculus on $\LG$.  The study of related problems in the theory of
degeneracy loci \cite{PR, LP, KT1}, Arakelov theory \cite{T1}, and quantum
cohomology \cite{KT2}, however, required representatives for the
Schubert classes in the Borel presentation of $\HH^*(\LG(n,2n))$, and
hence in the ring $\Lambda_n$. The answer was provided by Pragacz and
Ratajski's theory \cite{PR} of $\wt{Q}$-polynomials
$\wt{Q}_\la(\XX_n)$, which were extended to the $\wt{Q}$-functions
$\wt{Q}_\la(\XX)$ in \cite[\S 1.1]{T2}.  For each strict partition
$\la$, $\wt{Q}_\la(\XX)$ and $\wt{Q}_\la(\XX_n)$ are defined by the
raising operator expressions
\[
\wt{Q}_\la(\XX):=\RR\, e_\la(\XX) \ \ \mathrm{and} \ \
\wt{Q}_\la(\XX_n):=\phi_n(\wt{Q}_\la(\XX))=\RR\, e_\la(\XX_n),
\]
where as usual $e_\la:=\prod_i e_{\la_i}$. The {\em geometrization}
of the Schur $Q$-functions $Q_\la(X)$ is then displayed in the diagram
\[
\xymatrix@M=6pt{
Q_\la(X) \ar@{|->}[r]^{\pi} & \wt{Q}_\la(\XX) 
\ar@{|->}[d]^{\phi_n} \ar@{|->}[r] & \sigma_\la \ar@{|->}[d] \\
&  \wt{Q}_\la(\XX_n)  \ar@{|->}[r]^{\ \, \psi_0} & [X_\la]. }
\]
In other words, by geometrization we mean the choice of substitution
$Q_\la(X)\mapsto \wt{Q}_\la(\XX)$ shown above, which lifts the
ring homomorphisms
\[
\Gamma \stackrel{\pi}\lra \Lambda/\I \stackrel{\phi_n}\lra \Lambda_n/\I_n
\]
to maps of abelian groups
\[
\Gamma \stackrel{\pi}\lra \Lambda \stackrel{\phi_n}\lra \Lambda_n.
\]

The next step in the story was to extend the above picture to the
entire Weyl group $W_n$, and thus obtain a theory of symplectic
Schubert polynomials $\CS_w(\XX_n)$ for $\HH^*(\IF_n,\Z)$. In a
fundamental paper which built on the work of Lascoux-Sch\"utzenberger
\cite{LS} and Pragacz \cite{P}, Billey and Haiman \cite{BH} found the
combinatorially explicit family of type C Schubert polynomials
$\CS_w(X\,;Y)$ of \S \ref{spsdds}. These objects are actually formal
power series, realized as nonnegative integer linear combinations of
products $Q_\la(X)\AS_\om(Y)$ of Schur $Q$-functions and type A
Schubert polynomials. The $\CS_w(X\,;Y)$ for $w\in W_\infty$ form a
$\Z$-basis of a ring $\Gamma[Y]$ isomorphic to $\IH(\IF)$, and map
to the stable Schubert classes $\sigma_w$ under this isomorphism.

The problem with the Billey-Haiman power series $\CS_w(X\,;Y)$ was
that they were not related in \cite{BH} to the Borel presentation
(\ref{borpre}) in a way that retained combinatorial control over their
coefficients. The $\lambda$-ring substitution used in \cite[\S 2]{BH}
involved the odd power sums, and led to Schubert polynomials in the
root variables $\x_i$ which were quite complicated (see \cite[\S
  7]{FK} for a discussion of this). In 2006, motivated in part by an
application to arithmetic intersection theory, the author resolved
this issue by constructing a family of symplectic Schubert polynomials
$\CS_w(\XX_n)$ which are a {\em geometrization} of the $\CS_w(X\,;Y)$,
in the same sense that the polynomials $\wt{Q}_\la(\XX_n)$ are a
geometrization of the power series $Q_\la(Y)$.

Here are the details of this work, which eventually appeared in
\cite{T2}.  Let $\J$ be the ideal of $\Lambda[\XX]$ generated by the
elementary symmetric functions $e_p(\XX^2)$ for $p\geq 1$. The map
$\phi_n:\Lambda[\XX] \to \Z[\XX_n]$ sends $\J$ to the ideal $\J_n =
(e_1(\XX_n^2),\ldots,e_n(\XX_n^2))$. Define an isomorphism $\pi:
\Gamma[Y] \to \Lambda[\XX]/\J$ by setting $\pi(q_i(X)) := e_i(\XX)$
and $\pi(y_i):=-\x_i$ for all $i\geq 1$. We then have a diagram
\[
\xymatrix@M=6pt{
\Gamma[Y] \ar[r]^{\pi\ \,} & \Lambda[\XX]/\J 
\ar[d]^{\phi_n} \ar[r] & \IH(\IF) \ar[d] \\
 & \Z[\XX_n]/\J_n \ar[r]^{\psi} & \HH^*(\IF_n) }
\]
where the horizontal arrows are again ring isomorphisms. The map
$\psi$ is determined by $\psi(\x_i)= -c_1(E_{n+1-i}/E_{n-i})$ for
$1\leq i \leq n$.  Given any $w\in W_n$, apply equations
(\ref{dbleC2}) and (\ref{CStan0}) to write
\[
\CS_w(X\,;Y) = \sum_{v,\om,\la}e^v_\la \,Q_\la(X)\AS_\om(Y)
\]
where the sum is over all reduced factorizations $v\om=w$ and strict
partitions $\la$ such that $\om\in S_n$ and $|\la|=\ell(v)$.
Following \cite[\S 2.2]{T2}, define the symplectic Schubert
polynomials $\CS_w(\XX_n)$ and power series $\CS_w(\XX)$ by the
equations
\[
\CS_w(\XX):= \sum_{v,\om,\la}e^v_\la \,\wt{Q}_\la(\XX)\AS_\om(-\XX_n)
\ \ \mathrm{and} \ \ 
\CS_w(\XX_n):= \phi_n(\CS_w(\XX)).
\]
The {\em geometrization} of the Schubert polynomials $\CS_w(X\,;Y)$ is 
then displayed in the diagram 
\[
\xymatrix@M=6pt{
\CS_w(X\,;Y) \ar@{|->}[r]^{\ \ \pi} & \CS_w(\XX) 
\ar@{|->}[d]^{\phi_n} \ar@{|->}[r] & \sigma_w \ar@{|->}[d] \\
&  \CS_w(\XX_n)  \ar@{|->}[r]^{\ \psi} & [X_w]. }
\]

\begin{example}
\label{linearSP}
The linear Schubert polynomials $\CS_{s_i}$ are indexed by the simple
reflections $s_i$ in $W_\infty$. For each $i\in \N_0$, we have
\[
\CS_{s_i}(X\, ; Y) = q_1(X)+\AS_{s_i}(Y) = 2\left(\sum_{j=1}^{\infty} x_j\right)
+(y_1+\cdots + y_i)
\]
for the Billey-Haiman polynomials, while
\[
\CS_{s_i}(\XX) = \sum_{j=i+1}^{\infty}\x_j \quad \mathrm{and} \quad 
\CS_{s_i}(\XX_n) = \begin{cases} \x_{i+1}+\cdots + \x_n &
\text{if $i<n$}, \\  
0 & \text{otherwise}.
\end{cases}
\]
\end{example}

There remained one missing ingredient to fully solve the problem of
Schubert polynomials in the classical Lie types: define double
versions of the Billey-Haiman polynomials, and extend the above
picture to that setting. Fortunately, there was progress in this
direction, as in 2005 Ikeda \cite{I} had shown how Ivanov's factorial
Schur $Q$-functions \cite{Iv3} may be used to represent the
torus-equivariant Schubert classes in the equivariant cohomology of
the Lagrangian Grassmannian. At the March 2007 workshop on Schubert
calculus in Banff, the author suggested to Ikeda that he should work
on creating a double version of the Billey-Haiman theory. Ikeda was
also informed about the author's theory of symplectic Schubert
polynomials $\CS_w(\XX_n)$, and asked to include an extension of the
substitution $\CS_w(X\,;Y) \mapsto \CS_w(\XX_n)$ to the double case,
as it was important for geometric applications. In 2008, the work
\cite{IMN1} was announced, which used localization techniques as in
\cite{I, IN} to construct the required theory of polynomials
$\CS_w(X\,;Y,Z)$.  Moreover, the information necessary to extend the
author's geometrization of the $\CS_w(X\,;Y)$ to equivariant
cohomology was provided in \cite[\S 10]{IMN1}. This was done
explicitly in 2009, with the announcement of our paper \cite{T5}, 
and simultaneously generalized, to deal with the degeneracy loci 
coming from any isotropic partial flag variety. 

Following \cite[\S 4]{T5}, to describe how the $\CS_w(X\,;Y,Z)$
represent degeneracy loci of vector bundles, recall from \cite{Gr} and
\cite{T6} that these loci and their cohomology classes pull back from
the Borel mixing space $BM_n:=BB\times_{B\Sp_{2n}}BB$, which is the
universal case of (\ref{degdef}) and the degeneracy locus problem. Let
$$\IH(BM):=\lim_{\longleftarrow}\HH^*(BM_n,\Z)$$ be the stable
cohomology ring of $BM_n$.  Define the supersymmetric functions
$e_p(\XX/\YY)$ for $p\geq 0$ by
\[
e_p(\XX/\YY):=\sum_{j=0}^pe_j(\XX)h_{p-j}(\YY),
\]
set $\wt{\Lambda}:=\Z[e_1(\XX/\YY),e_2(\XX/\YY),\ldots]$, and let
$\wt{\J}$ be the ideal of $\wt{\Lambda}[\XX,\YY]$ generated by the
fundamental relations
\[
e_p^2(\XX/\YY) + 2\,\sum_{j=1}^p(-1)^j e_{p+j}(\XX/\YY)e_{p-j}(\XX/\YY)
\]
for all $p\geq 1$. Define an isomorphism 
\[
\wt{\pi}:\Gamma[Y,Z] \to \wt{\Lambda}[\XX,\YY]/\wt{\J}
\] 
by setting $\wt{\pi}(q_i(X)) := e_i(\XX/\YY)$, $\wt{\pi}(y_i):=-\x_i$,
and $\wt{\pi}(z_i):=\y_i$, for each $i\geq 1$.  Moreover, let
$\wt{\phi}_n:\wt{\Lambda}[\XX,\YY] \to \Z[\XX_n,\YY_n]$ be the map
determined by $\x_i\mapsto \x_i$ and $\y_i\mapsto \y_i$ for $i\leq n$,
while $\x_i\mapsto 0$ and $\y_i\mapsto 0$ for $i>n$, and set
$\wt{\J}_n:=\phi_n(\wt{\J})\subset \Z[\XX_n,\YY_n]$. (It is not hard
to show that $\wt{\J}_n$ is equal to the ideal of $\Z[\XX_n,\YY_n]$
generated by the differences $e_p(\XX_n^2)-e_p(\YY_n^2)$ for $1\leq p
\leq n$.)  We then have a commutative diagram of rings
\[
\xymatrix@M=6pt{
\Gamma[Y,Z] \ar[r]^{\wt{\pi}\ \,} & \wt{\Lambda}[\XX,\YY]/\wt{\J} 
\ar[d]^{\wt{\phi}_n} \ar[r] & \IH(BM) \ar[d] \\
 & \Z[\XX_n,\YY_n]/\wt{\J}_n \ar[r]^{\wt{\psi}} & \HH^*(BM_n,\Z) }
\]
extending the previous ones, where again the horizontal maps are 
isomorphisms. The map $\wt{\psi}$ satisfies
$\wt{\psi}(\x_i)= -c_1(E_{n+1-i}/E_{n-i})$ and 
$\wt{\psi}(\y_i)= -c_1(F_{n+1-i}/F_{n-i})$ for each $i$ with $1\leq i \leq n$. 

Given $w\in W_n$, apply equations (\ref{dbleC2}) and
(\ref{CStan0}) to write
\[
\CS_w(X\,;Y,Z) = \sum_{u,v,\om,\la}e^v_\la \,\AS_{u^{-1}}(-Z)Q_\la(X)\AS_\om(Y)
\]
where the sum is over all reduced factorizations $uv\om=w$ and strict 
partitions $\la$ such that $u,\om\in S_n$ and $|\la|=\ell(v)$. Define the 
{\em supersymmetric $\wt{Q}$-function} $\wt{Q}_\la(\XX/\YY)$ by the 
equation
\[
\wt{Q}_\la(\XX/\YY):= \RR\, e_\la(\XX/\YY).
\]
and the {\em double symplectic Schubert polynomials} $\CS_w(\XX_n,\YY_n)$ and 
power series $\CS_w(\XX,\YY)$ by the equations
\[
\CS_w(\XX,\YY):= \sum_{u, v,\om,\la}e^v_\la \,\AS_{u^{-1}}(-\YY_n)
\wt{Q}_\la(\XX/\YY)\AS_\om(-\XX_n)
\]
and $\CS_w(\XX_n,\YY_n):= \wt{\phi}_n(\CS_w(\XX/\YY))$.  The {\em
  geometrization} of the double Schubert polynomials $\CS_w(X\,;Y,Z)$
is exhibited in the diagram
\[
\xymatrix@M=6pt{
\CS_w(X\,;Y,Z) \ar@{|->}[r]^{\ \, \wt{\pi}} & \CS_w(\XX,\YY) 
\ar@{|->}[d]^{\wt{\phi}_n} \ar@{|->}[r] 
& \text{$\dis \lim_{\longleftarrow}[\X_w]$} \ar@{|->}[d] \\
&  \CS_w(\XX_n,\YY_n)  \ar@{|->}[r]^{\ \quad  \wt{\psi}} & [\X_w]. }
\]
As before, we observe that the substitutions 
$Q_\la(X)\mapsto \wt{Q}_\la(\XX/\YY)$, $y_i\mapsto -\x_i$, and 
$z_i\mapsto \y_i$ lift the ring homomorphisms
\[
\Gamma[Y,Z] \stackrel{\wt{\pi}}\lra \wt{\Lambda}[\XX,\YY]/\wt{\J}
 \stackrel{\wt{\phi}_n}\lra \Z[\XX_n,\YY_n]/\wt{\J}_n
\]
to maps of abelian groups
\[
\Gamma[Y,Z] \stackrel{\wt{\pi}}\lra \wt{\Lambda}[\XX,\YY] 
\stackrel{\wt{\phi}_n}\lra \Z[\XX_n,\YY_n].
\]

Since the variables $\{-\x_i\}$ and $\{-\y_i\}$ for $1\leq i \leq n$
represent the Chern roots of the isotropic vector bundles in (\ref{seqs}),
it is now a simple matter to translate the formulas in this paper into
Chern class formulas for degeneracy loci. For each $r\geq 0$, define
$c_r(E-E_i-F_j)$ by the equation of total Chern classes
\[
c(E-E_i-F_j):=c(E)c(E_i)^{-1}c(F_j)^{-1}.
\]
Then the {\em geometrization map}
$\omega_n:=\wt{\psi}\wt{\phi}_n\wt{\pi}$ sends $q_r(X)$ to
$c_r(E-E_n-F_n)$ for each $r\geq 0$, and $Q_\la(X)$ to
$Q_\la(E-E_n-F_n)$ for every strict partition $\la$. 

More generally, following \cite[Thm.\ 3]{T5}, the substitution which
maps the theta polynomial $\Ti_\la(X\,;Y_{(k)})$ to
$\Ti_\la(E-E_{n-k}-F_n)$ for each $k$-strict partition $\la$ is
applied to treat the degeneracy loci which come from any symplectic
partial flag variety (see also \cite[Remark 4]{T6}). Computing the
image of formula (\ref{dbCSsplitting}), we thus obtain that the class
of the degeneracy locus $\X_w$ in $\HH^*(M)$ is equal to
\[
\sum_{\underline{\la}} f^w_{\underline{\la}}\,
s_{\la^1}(F_{n+b_{q-1}} - F_{n+b_q})\cdots \Ti_{\la^q}(E-E_{n-a_1}-F_n)
\cdots s_{\la^{p+q-1}}(E_{n-a_{p-1}}-E_{n-a_p})
\]
with the coefficients $f^w_{\underline{\la}}$ given by equation
(\ref{dbfdef0}).

The polynomials ${}^kc^r_p$ defined in (\ref{ceq}) are particularly
useful to work with, as we have the Chern class equation
\[
\omega_n({}^kc^r_p)= c_p(E-E_{n-k}-F_{n+r})
\]
(compare with \cite[Eqn.\ (31)]{TW}). For the top Schubert polynomial
$\CS_{w_0}(X\, ; Y,Z)$, equation (\ref{Ctopeq}) maps to the Pfaffian
formula
\begin{equation}
\label{cp}
[\X_{w_0}] = Q_{\delta_n+\delta_{n-1}}
(E-E_{(1,2,\ldots,n)}-F_{(1,2,\ldots,n)}) 
\end{equation}
in $\HH^*(M)$. Following \cite[Cor.\ 1]{TW}, the Chern polynomial in
(\ref{cp}) is defined as the image of the polynomial $R^\infty \,
c_{\delta_n+\delta_{n-1}}$ under the $\Z$-linear map which sends the
noncommutative monomial $c_\al$ to $\prod_jc_{\al_j}(E-E_j-F_j)$,
for every integer sequence $\al$. The geometric substitutions of this
section can thus be used to relate some of the equations for Schubert
polynomials found in the present paper to the Chern class formulas
in \cite{Ka, AF1, AF2} and elsewhere.

\end{document}